\documentclass{elsarticle}
\makeatletter
\def\ps@pprintTitle{%
 \let\@oddhead\@empty
 \let\@evenhead\@empty
 \def\@oddfoot{\centerline{\thepage}}%
 \let\@evenfoot\@oddfoot}
\makeatother
\usepackage{amsmath,latexsym,amsfonts,amssymb,mathrsfs}
\usepackage{graphicx,verbatim}

\usepackage{xcolor,soul,ulem,empheq}
\usepackage{url}
\numberwithin{equation}{section}

\newtheorem{theorem}{Theorem}[section]

\newtheorem{lemma}{Lemma}[section]
\newtheorem{corollary}{Corollary}[section]
\newtheorem{proposition}{Proposition}[section]
\newtheorem{remark}{Remark}[section]

\newenvironment{proof}{\vspace{0.3cm}\hfill\\% 
\noindent\textbf{Proof}}{$\Box$
\vspace{0.2cm}\\}

\newcommand{\be}{\begin{equation}}
\newcommand{\ee}{\end{equation}}
\newcommand{\ba}{\begin{array}}
\newcommand{\ea}{\end{array}}
\newcommand{\beas}{\begin{eqnarray*}}
\newcommand{\eeas}{\end{eqnarray*}}
\newcommand{\bea}{\begin{eqnarray}}
\newcommand{\eea}{\end{eqnarray}}

\newcommand{\inprod}[2]{\left\langle{#1},{#2}\right\rangle}
\newcommand{\calA}{\mathcal{A}}
\newcommand{\calB}{\mathcal{B}}
\newcommand{\calK}{\mathcal{K}}
\newcommand{\calL}{\mathcal{L}}
\newcommand{\calN}{\mathcal{N}}

\newcommand{\R}{\mathbb{R}}
\newcommand{\bbA}{\mathbb{A}}
\newcommand{\bbP}{\mathbb{P}}
\newcommand{\bbE}{\mathbb{E}}
\newcommand{\N}{\mathbb{N}}
\newcommand{\mS}{\mathbb{S}}
\newcommand{\cH}{\mathcal{H}}
\newcommand{\frakF}{\mathfrak{F}}
\newcommand{\bx}{{\bf x}}
\newcommand{\by}{{\bf y}}
\newcommand{\what}{\widehat}
\newcommand{\wtd}{\widetilde}
\newcommand{\sumlm}{\sum_{\ell=0}^\infty \sum_{m=-\ell}^\ell}

\newcommand{\red}[1]{{#1}}
\newcommand{\eqhl}[1]{{#1}}
\newcommand{\mulop}{\mathscr{M}}
\newcommand{\YK}[1]{{#1}}

\soulregister\cite7
\soulregister\ref7
\soulregister\pageref7

\begin{document}
\begin{frontmatter}

\title{A non-uniform discretization of stochastic heat equations with multiplicative noise on the unit sphere}
%\author{Yoshihito Kazashi\corref{cor1}\fnref{label2} and Q. T. Le Gia}
\author{Yoshihito Kazashi\corref{cor1}}
 \ead{y.kazashi@unsw.edu.au}
 \cortext[cor1]{Corresponding author.}
\author{Quoc T. Le Gia\corref{cor2}}
 \ead{qlegia@unsw.edu.au}
% \cortext[cor2]{the corresponding author}


\address{School of Mathematics and Statistics, University of New South Wales, Sydney, NSW 2052, Australia}

\begin{abstract}
We investigate a discretization of a class of stochastic heat equations on the unit sphere with multiplicative noises. 
A spectral method is used for the spatial discretization and the truncation of the Wiener process, while an implicit Euler scheme with non-uniform steps is used for the temporal discretization. Some numerical experiments inspired by Earth's surface temperature data analysis GISTEMP provided by NASA are given.
\end{abstract}


\end{frontmatter}
%------------------------------------------------------------------
\section{Introduction}
Let $\mS^2$ be the unit sphere in the Euclidean space $\R^{3}$,
that is $$\mS^2 = \{\bx\in \R^3:  |\bx| = 1 \},$$ where $|\cdot|$ 
denotes the usual Euclidean norm.
We consider the following stochastic heat equation
\begin{equation}\label{equ:main}
\begin{aligned}
  \mathrm{d}X(t) &= \Delta^\ast X(t) \mathrm{d}t + B(X(t)) \mathrm{d}W(t), \\
   X(0) &= \xi, \qquad t \in [0,1],
\end{aligned}
\end{equation}
on the Hilbert space $H = L^2(\mS^2)$, the space of equivalence classes of square integrable functions. 
Here $\xi \in H$ is the {deterministic} initial value, and $\Delta^\ast$
denotes the Laplace--Beltrami operator on $\mS^2$.
Under suitable assumptions on $B$, a mild solution $X = (X(t))_{t\in [0,1]}$
of \eqref{equ:main} exists and is uniquely determined as a continuous
process with values in $H$ (see, e.g., Da Prato and Zabczyk \cite{DaPZab14}). 

{F}or a bounded domain in $\R^d$ and a standard scalar Wiener process, 
numerical algorithms that solve general stochastic evolution equations on Hilbert spaces 
were constructed and analyzed first
in the work of Grecksch and Kloeden \cite{GreKlo96}. 
{Gy\"{o}ngy and Nualart \cite{GyoNua97} also considered an implicit scheme
for stochastic parabolic partial differential equations (PDEs) over the unit interval driven by space-time white noise.}
%\footnote{I am not sure if this is accurate. Grecksch and Kloeden \cite{GreKlo96} considered the scalar Wiener process, and it doesn't seem \cite{GyoNua97} did anything general.}.
%\footnote{{I do not think Gy\"{o}ngy and Nualart \cite{GyoNua97} considered general stochastic evolution equations on Hilbert spaces, and they considered the space-time white noise, not the standard scalar Wiener process.}}
Further contributors to the problem include 
Allen, Novosel, and Zhang \cite{AllNovZha98}, 
Gy\"{o}ngy\cite{Gyo99}, 
Shardlow \cite{Sha99}
Davie and Gaines \cite{DavGai01}, 
Du and Zhang \cite{DuZha02},  
Kloeden and Shott \cite{KloSho01}, 
Hausenblas \cite{Hau02,Hau03},
Lord and Rougemon \cite{LorRou04}, Yan \cite{Yan04,Yan05}, and M\"{u}ller-Gronbach and Ritter \cite{GroRit07,GroRit07a}. %YK: this sentence seems to be out of place: M\"{u}ller-Gronbach and Ritter \cite{GroRit07,GroRit07a} proposed a discretization scheme for the heat equation on the unit cube $[0,1]^d$ which allow different time steps for different eigenspace of the covariance operator. 

Recently, using a characterization of $Q$-Wiener processes (where $Q$
is the covariance operator) {on the sphere that has a rotationally invariant covariance as a random field at a fixed time}, Lang and Schwab
\cite{LanSch14} considered a numerical scheme for a special case
of \eqref{equ:main} with $B(X)$ being the identity, {i.e., equations with the additive noise.}

In this work, we consider the equations on the sphere with the multiplicative noise. 
Following \cite{LanSch14}, we consider $Q$-Wiener processes that have a rotationally invariant covariance function. A natural question {arising} would be whether or not the invariance propagates. Considering a class of affine noise, we derive an equation of second moment, and show a characterization of the invariance of the covariance function under rotation. 

{We will \YK{further} study an It\^{o}--Galerkin method when $B(X)$ is assumed to satisfy certain growth conditions, and consider a non-uniform temporal discretization, and establish a convergence rate.} Our work also can be seen as an extension of the works by M\"{u}ller-Gronbach and Ritter \cite{GroRit07,GroRit07a}, {who proposed a discretization 
scheme for the heat equation on the unit cube $[0,1]^d$ which allow different time steps for different eigenspace of the covariance operator}, to the spherical case. 
We remark that this is a non-trivial task. Their proofs that validate the non-uniform time step do not seem to be easily generalizable to a general Hilbert space setting:  
in the argument in \cite{GroRit07,GroRit07a}, the eigenfunctions of the Laplace operator on the cube with the Dirichlet condition being uniformly bounded is repeatedly used in the proof, further the integration by parts on $[0,1]$, 
which uses the zero Dirichlet boundary condition, is crucial. 
On the sphere, we have neither of the properties. 
Upon the normalization to make them orthonormal on $L^2(\mS^2)$, the magnitude of spherical harmonics, 
the eigenfunctions of the Laplace--Beltrami operator, grows as the degree of the polynomial goes up. Further, on the sphere, we lack a convenient first order derivative that corresponds to the usual derivative on $[0,1]$. These difficulties are treated by exploiting the properties of spherical harmonics.  %First, we will define theWiener process on the sphere then discretize the stochastic partial differential equation (SPDE) by a variant of the implicit Euler--Maruyama method in time/It\^{o}--Galerkin in space. We analyze the convergence rate of the approximation.

The paper is organized as follows. In Section~\ref{sec:prelim}, we 
review necessary facts on function spaces,
random fields and Brownian motions on the unit sphere.
In Section~\ref{sec:SDEs}, we then introduce stochastic evolution equations on the sphere, and discuss the isotropy of the solution. 
Section~\ref{sec:discrete} deals with the discretization of the SDEs using
an Euler--Maruyama scheme. %The mean squared errors between the exact solution the numerical solutions are analysis in Section~\ref{sec:err anal}.
{Error bounds are stated in Section~\ref{sec:err anal}.}
Finally some numerical
results based on Earth's surface temperature analysis GISTEMP data provided
by NASA Goddard Institute for Space Science
will be presented in Section~\ref{sec:num}.
%--------------------------------------------------------------------------------------
\section{Preliminaries}\label{sec:prelim}
\subsection{Spherical harmonics and function spaces}

Let ${H:=} L^2(\mS^2)$ be the space of the equivalence class{es} of the square integrable functions on the unit sphere,
which is equipped with the following standard inner product
\begin{equation}\label{real inner} 
  \inprod{f}{g} := \int_{\mS^2} f(\bx) g(\bx) \mathrm{d}\varsigma(\bx),
\end{equation}
where $\mathrm{d}\varsigma$ 
is the surface measure of $\mS^2$. We write $\|f\|:=\sqrt{\inprod{f}{f}}$ for $f\in L^2(\mS^2)$. In spherical 
coordinates, for a point $\bx \in \mS^2$, we have the parametrization
$$\bx = (\sin\vartheta \cos \varphi, \sin\vartheta \eqhl{\sin\varphi}, \cos\vartheta)$$ 
for $\vartheta \in [0,\pi]$ and $\varphi \in [0,2\pi)$, where
for $\theta\in\{0,\pi\}$ we let $\varphi=0$. 
Further, we let $\mathrm{d}\varsigma(\bx) = \sin \vartheta \mathrm{d}\vartheta \mathrm{d} \varphi$.

The space $L^2(\mS^2)$ admits the spherical harmonics as a complete orthonormal system. 
Spherical harmonics are the restrictions to $\mS^2$ of homogeneous polynomials $Y(\bx)$ 
in $\R^{3}$ which satisfy $\Delta Y(\bx) = 0$, where ${\Delta}$ is the Laplace
operator {for functions on} $\R^{3}$. The space of all spherical harmonics of
degree $\ell$ on $\mS^2$, denoted by $\cH_\ell$, has an orthonormal
basis
$
  \{Y_{\ell m}:  m = -\ell,\ldots,\ell\},
$
and 
$
\mathrm{span} \{ Y_{\ell m} : \ell\ge 0, |m| \le \ell \}
$ 
is dense in  $L^2(\mS^2)$.

The explicit formula for $Y_{\ell m}$ is given by
\begin{align}
Y_{\ell m}(\vartheta,\varphi)
 =
\begin{cases}
\eqhl{ 
  \sqrt{2}\sqrt{ \frac{2\ell+1}{4\pi} \frac{(\ell-|m|)!}{(\ell+|m|)!}}
   P^{|m|}_{\ell}(\cos\vartheta) \sin(|m| \varphi)
}
   & \text{for }m=-\ell,\dotsc,-1 \\
% Y_{\ell m}(\vartheta,\varphi)
\eqhl{ 
\sqrt{ \frac{2\ell+1}{4\pi}} %\frac{(\ell-m)!}{(\ell+m)!}}
   P^0_{\ell}(\cos\vartheta)
}
   & \text{for }m=0 \\
% Y_{\ell m}(\vartheta,\varphi)
\eqhl{
 \sqrt{2}\sqrt{ \frac{2\ell+1}{4\pi} \frac{(\ell-m)!}{(\ell+m)!}}
   P^m_{\ell}(\cos\vartheta) \cos( m \varphi)
}
   & \text{for }m=1,\dotsc,\ell{,} \\
\end{cases}
\end{align}
where $P^m_{\ell}$ is the associated Legendre polynomial \red{of degree $\ell$ and order $m$}, given by
\[
  P^m_{\ell}(x) = (-1)^m (1-x^2)^{m/2} \frac{d^m}{dx^m} P_\ell(x), \quad m \ge 0,
  \qquad\text{ for }\ x\in[-1,1],
\]
where $P_\ell$ is the Legendre \red{polynomial}.
Thus
\[
 \inprod{Y_{\ell m}}{Y_{\ell' m'}} =  
\int_{\mS^2} Y_{\ell m}(\bx) Y_{\ell' m'}(\bx) \mathrm{d}\varsigma(\bx) = 
  \delta_{\ell \ell'}\delta_{m m'},
\]
where $\delta_{\ell \ell'}$ is the Kronecker symbol. 

The spherical harmonics of degree $\ell$ satisfy the following addition
theorem \cite{Mul66}
\begin{equation}\label{addition}
 \sum_{m=-\ell}^\ell Y_{\ell m}(\bx) Y_{\ell m}(\by) = 
\frac{2\ell+1}{4\pi} P_\ell(\bx \cdot \by).
\end{equation}
%and the surface
%measure on $\mS^2$ is  $$\mathrm{d}\varsigma(\bx)=\sin\vartheta d\vartheta d\varphi.$$ 

The spherical harmonics are the eigenfunctions of the Laplace--Beltrami
operator $\Delta^\ast$ with eigenvalues $-\mu_{\ell} = -\ell(\ell+1)$ for 
$\ell=0,1,2,\dots$.
In other words,
\[
  \Delta^\ast Y_{\ell m} = - \mu_\ell Y_{\ell m}.
\]
A more detailed discussion on spherical harmonics in $\R^{d+1}$ for $d\ge 2$ 
can be found in \cite{Mul66}. 
We define the Sobolev space $H^1$ on the sphere $\mS^2$ as the domain of $({1-}\Delta^\ast)^{\frac12}$:
\[
  H^1 := \left\{ h \in L^2(\mS^2) : \|h\|^2_{H^1} =\sumlm (1+\mu_\ell) 
  \inprod{h}{Y_{\ell m}}^2 < \infty \right\}.
\]
%and its subspace
%\[
% H^1_0 := \left\{ h \in L^2(S) : \int_{\mS^2} h dS = 0,\;\;
% \|h\|^2_{H^1_0} = \sumlm \mu_\ell
%  \inprod{h}{Y_{\ell m}}^2 < \infty \right\}.
%\]
%Since $\mu_\ell  \ge 2$ for all $\ell>0$ we 
%have the following inequality
%\begin{equation}\label{equ:poincare}
% \|h\|^2_{H^1} \le 2 \|h\|^2_{H^1_0} \text{ for all } h \in H^1_0
%\end{equation}
%-----------------------------------------------------------------
%-----------------------------------------------------------------
\subsection{Isotropic Gaussian random fields on the sphere}
In order to define Wiener processes properly on the sphere, 
firstly {we discuss} random fields defined on spheres.
Random fields on spheres arise in modelling the cosmic microwave
background (CMB) \cite{MarPec11}, modeling Saharan dust particles \cite{NouMuiRai03}, 
feldspar particles \cite{VeiNouKah06}, ice crystals \cite{NouMcF04} etc. 

To define a random field on $\mS^2$, 
let $(\Omega,\calA,\bbP)$ be a probability space and let
$\calB(\mS^2)$ be the Borel $\sigma$-algebra of $\mS^2$ {with respect to the usual spherical metric topology}.
A $\calA \otimes \calB(\mS^2)$-measurable mapping
$T: \Omega \times \mS^2 \rightarrow \R$ is called a (product measurable) 
real-valued random field on the unit sphere.% \YK{We further let $L^2(\Omega;H)$ denote the Bochner space of square integrable $\bbP$-strongly measurable functions.}

{A} random field is called strongly isotropic if, for all $k \in \N$,
$\bx_1,\ldots,\bx_k \in \mS^2$, and for {all} $g \in SO(3)$, (here $SO(3)$ denotes
the group of rotations on $\mS^2$), the multivariate
random variables $(T(\bx_1),\ldots,T(\bx_k))$ and
$(T(g\bx_1),\ldots,T(g \bx_k))$ have the same law.

It is called $n$-weakly isotropic for $n\ge 2$ if
$\bbE(|T(\bx)|^2) < \infty$ for all $\bx \in \mS^2$ and if for
$1 \le k \le n$, $\bx_1,\ldots,\bx_k \in \mS^2$ and $g \in SO(3)$,
\[
 \bbE(T(\bx_1) \cdots T(\bx_k)) = \bbE( T(g\bx_1) \cdots T(g\bx_k)).
\]

Furthermore, it is called Gaussian if for all $k \in \N$, $\bx_1,\ldots,\bx_k$
the random variable $(T(\bx_1),\ldots,T(\bx_k))$ is multivariate Gaussian
distributed, or equivalently, if $\sum_{i=1}^k a_i T(\bx_i)$ is a normally distributed
random variable for all $a_i \in \R$, $i=1,\ldots,k$ {for all $k\in\N$}. 

%In the remainder of the paper, we focus on real-valued Gaussian random fields on $\mS^2$.
{For Gaussian random fields, we have the following characterization of the strong isotropy.}
\begin{proposition}[Proposition 5.10 in \cite{MarPec11}]\label{prop:MP51}
Let $T$ be a Gaussian random field on $\mS^2$. Then, $T$ is strongly isotropic if and only
if T is $2$-weakly isotropic.
\end{proposition}
The following result is immediately obtained, 
which was originally considered for 
the spherical harmonics of the complex form in \cite[Theorem 5.13]{MarPec11}.
\red{
We note that the proof of} \cite[Theorem 5.13]{MarPec11} \red{relies on} \YK{Theorem 5.5} \red{in the same book, which relies on a group representation theorem and an orthonormal system in $L^2(\mS^2)$. }\YK{If} \red{we replace the complex inner product with a real one as in (\ref{real inner}) and the complex spherical harmonics with real spherical harmonics, then the following theorem is obtained.
}
\begin{theorem}%[Theorem 5.13 in \cite{MarPec11}]
\label{thm:MP513}
Let $T$ be a 2-weakly isotropic random field on $\mS^2$, then the following
statements hold true:
\begin{enumerate}
\item  $T$ satisfies
\[
  \int_{\mS^2} T(\bx)^2 \mathrm{d}\varsigma(\bx) < \infty,
  \quad\text{\YK{almost surely.}}
\]
\item $T$ admits a Karhunen--Lo\`{e}ve expansion
\begin{equation}\label{equ:KLofT}
   T = \sumlm \alpha_{\ell m} Y_{\ell m}, \qquad
   \alpha_{\ell m} = \int_{\mS^2} T(\by) Y_{\ell m} \mathrm{d}\varsigma(\by)
   \YK{,}
\end{equation}
\YK{where the convergence is both in the sense of the following:}
\begin{enumerate}
\item The series expansion \eqref{equ:KLofT} converges in 
$L^2(\Omega \times \mS^2; \R)$, that is, 
\[
 \lim_{L \rightarrow \infty}
 \bbE \left( \int_{\mS^2} (T(\by) - 
 \sum_{\ell=0}^L  \sum_{m=-\ell}^\ell \alpha_{\ell m} Y_{\ell m} (\by))^2
 \mathrm{d}\varsigma(\by)\right) = 0.
\]
\item The series expansion \eqref{equ:KLofT} converges in $L^2(\Omega;\R)$
for all $\bx \in \mS^2$, i.e.,
\[
\lim_{L \rightarrow \infty}
\bbE \left( (T(\bx) - 
\sum_{\ell=0}^L  \sum_{m=-\ell}^\ell \alpha_{\ell m} Y_{\ell m} (\bx) )^2 \right) = 0,\quad \text{\YK{for all $\bx\in \mS^2$.}}
\]
\end{enumerate}
\end{enumerate}
\end{theorem}

Let $T$ be a strongly isotropic random field on $\mS^2$, then {by adapting Remark 6.4 and Equation (6.6) in
\cite{MarPec11} to the real spherical harmonics}, the
\YK{collection} $\bbA = (\alpha_{\ell m}, \ell \in \N_0, m=-\ell,\ldots,\ell)$ are,
except for $\alpha_{00}$, centered random variables, i.e. $\bbE(\alpha_{\ell m})=0$
for all $\ell \in \N$ and $m=-\ell,\ldots,\ell$. Furthermore, they
are {real}-valued random variables that satisfy
\begin{equation}
\bbE(\alpha_{\ell m} \alpha_{\ell' m'})= 
A_{\ell} \delta_{\ell \ell'} \delta_{m m'}, 
\quad \ell,\ell' \in \N, \;\;|m| \le \ell, \;\;|m'|\le \ell'.
\end{equation}
For $\alpha_{00}$, it holds that
\begin{equation}
 \bbE(\alpha_{00} \alpha_{\ell m}) = (A_0 + \bbE({\alpha_{00}})^2) 
 \delta_{0 \ell} \delta_{0 m}.
\end{equation}
The sequence of non-negative real numbers $(A_\ell, \ell \in \N_0)$ is
called the angular power spectrum of $T$. 
We note that
$\bbE(T)=\bbE(\alpha_{00})Y_{00}=\bbE(\alpha_{00})\sqrt{\frac{1}{4\pi}}$. 
Combining the results of Proposition~\ref{prop:MP51} and Theorem~\ref{thm:MP513}
we obtain the following corollary.

%\footnote{{But we consider the Wiener process and thus zero-mean.}}
\begin{corollary}
Let $T$ be a 2-weakly isotropic Gaussian random field on $\mS^2$. Then
$T$ admits the Karhunen--Lo\`{e}ve expansion
\[
  T = \sumlm \alpha_{\ell m} Y_{\ell m},
\]
where $\bbA = (\alpha_{\ell m}, \ell \in \N_0, m=-\ell,\ldots,\ell)$ is 
a \YK{family} of \YK{independent} real-valued Gaussian random variables
% with
\YK{such that $\alpha_{\ell m}\sim\calN(0,A_\ell)$ for $\ell>0$ while 
$\alpha_{00}\sim\calN( \bbE(T) 2\sqrt{\pi},A_0)$.}
%\begin{enumerate}
%\item $\bbA = (\alpha_{\ell m}, \ell\in \N_0, m=-\ell,\ldots,\ell)$ is 
%a \YK{family} of independent, {real}-valued Gaussian random variables.
%\item The elements of {$\bbA$} 
%are independent.
%\item The elements of {$\bbA$} follow {$\calN(0,A_\ell)$},
%i.e. \red{they are} normally distributed with mean $0$ and variance {$A_{\ell}$} 
%for $\ell \in \N$, while 
%{$\alpha_{00}$} is 
%$\calN( \bbE(T) 2\sqrt{\pi},A_0)$ distributed.
%\footnote{{(Comment: $\bbE(T)=\bbE(\alpha_{00})Y_{00}=\bbE(\alpha_{00})\sqrt{\frac{2\cdot0+1}{4\pi}}$).}}
%\item \sout{The elements of $\bbA$ with $m<0$ are deduced from those of $\bbA_{+}$ by the formulae}
%\begin{empheq}[box=\shadebox]{align}
%\Re{a_{\ell m}} = (-1)^m \Re{ a_{\ell, -m}},\quad
%\Im{a_{\ell m}} = (-1)^{m+1} \Im{ a_{\ell, -m} }.
%\end{empheq}
%\end{enumerate}

\end{corollary}
%\sout{Furthermore, by Proposition 6.8 in \cite{MarPec11}, the elements
%of $\bbA_{+}$ for $m \ne 0$ satisfy that $\Re{a_{\ell m}}$ and
%$\Im{a_{\ell m}}$ are symmetric random variables that are equal in
%law, uncorrelated, i.e. $\bbE(\Re{a_{\ell m}} \Im a_{\ell m}) = 0$ and
%that have variance}
%\begin{empheq}[box=\shadebox]{align}\label{equ:varsqr}
% \bbE((\Re a_{\ell m})^2) =  \bbE((\Im a_{\ell m})^2) = A_{\ell}/2.
%\end{empheq}
%---------------------------------------------------------
\subsection{$L^2(\mS^2)$-valued $Q$-Wiener process}

\begin{comment}
Consider a separable Hilbert space $H$ with inner product $\inprod{\cdot}{\cdot}$
and a self-adjoint, positive definite bounded linear operator $Q:H \rightarrow H$.
Moreover, let $(\frakF_t)_{t\ge 0}$ denote a right-continuous filtration in a 
complete probability space $(\Omega,\frakF,P)$. 

A family $(W(t,h))_{t \ge 0, h \in H}$ of real-valued random variables on
$(\Omega,\frakF,P)$ is called a cylindrical Brownian motion on $H$ with covariance
$Q$, if the following properties hold:
\begin{itemize}
\item[(a)] $(\inprod{Qh}{h}^{-1/2} W(t,h))_{t\ge 0}$ is a standard one-dimensional
Brownian motion w.r.t the filtration $(\frakF)_{t \ge 0}$ for every 
$h \in H \setminus \{0\}$.
\item[(b)] For all $c_1,c_2 \in \R$, $h_1,h_2 \in H$, and every $t \ge 0$,
\[
 W(t,c_1 h_1 + c_2 h_2) = c_1 W(t,h_1) + c_2 W(t,h_2) \mbox{ holds P-a.s.}
\]
\item[(c)] $(W(t,h))_{t \ge 0}$ is a martingale w.r.t 
$(\sigma(\{W(s,g):0 \le s\le t, g \in H\}))_{t\ge 0}$ for every $h \in H$.
\end{itemize}
\end{comment}
%{Without loss of generality, from now on we consider only centered Gaussian random fields (i.e. Gaussian fields with zero mean)}.
We now define {an $H$}-valued Wiener process that is an isotropic centered Gaussian 
random field for any fixed time $t$.

%{Moved this bit:}
In the following, we assume { $\{A_\ell\}_{\ell \ge 0}$ is a given sequence of {positive} real numbers
such that}
\begin{align}\label{equ:Aell summable}
\sum_{\ell=0}^{\infty}(2\ell+1)A_\ell<\infty.
\end{align}
Then, the covariance kernel of an isotropic centered 
%\footnote{{If centred, we must have $\bbE[\alpha_{00}]=0$ anyway?}} 
Gaussian random field $T$ is well defined, and
is given by the formula
\begin{equation}\label{def:KT}
\begin{aligned}
K_T(\bx,\by) := \bbE[T(\bx)T(\by)] &= \sum_{\ell=0}^\infty A_\ell
\sum_{m=-\ell}^{\ell} Y_{\ell m} (\bx) Y_{\ell m}(\by)\\
&= \sum_{\ell=0}^\infty
 A_\ell \frac{2\ell+1}{4\pi} P_\ell(\bx \cdot \by),
\end{aligned}
\end{equation}
which in turn ensures the existence of Gaussian random fields, see for example 
\cite{Bogachev.V.I_1998_Gaussian_Measures,Lifshits.M.A_1995_Gaussian_Random_Functions}.

Let $Q\colon H\to H$ be the integral operator associated with the covariance kernel \eqref{def:KT}, 
that
is, for an element $f\in H$,
\[
Q f(\bx) = \int_{\mS^2} K_T(\bx, \by) f(\by) \mathrm{d}\varsigma(\by), \quad \bx \in \mS^2.
\]
Then, we see that
\begin{align*}
  Q Y_{L,M}(\bx) &= \int_{\mS^2} K_T(\bx,\by) Y_{L,M}(\by) \mathrm{d}\varsigma{(\by)} \\ 
%  \sum_{\ell=0}^\infty \sum_{m=-\ell}^\ell
%  \bbE(\alpha_{L,M}(1) \alpha_{\ell,m}(1)) Y_{\ell,m}  \\
     &=A_L Y_{L,M}(\bx),  \quad L=0,1,2,\ldots; |M| \le L.
\end{align*}
and thus from \eqref{equ:Aell summable} $Q$ is 
of trace class with $\mathrm{Tr}(Q)=\sum_{\ell=0}^{\infty}(2\ell+1)A_\ell<\infty$.
\begin{comment}
A 1-dimensional Wiener process $\beta(t)$ is characterized by the following properties
\begin{itemize}
\item $\beta(0) = 0$.
\item The function  $t \mapsto \beta(t)$ is almost surely everywhere continuous,
\item $\beta(t)$ has independent increments with $\beta(t) - \beta(s) \sim \calN(0,t-s)$
(for $0\le s < t$), where $\calN(\mu,\sigma^2)$ denotes the normal distribution
with mean $\mu$, variance $\sigma^2$.
\end{itemize}
\end{comment}
%In view of previous sections\footnote{{I am not too sure if the characterization is in view of the previous section. Once we fix $Q$ as we have done, the representation $W(t,\bx) = 
%\sum_{\ell=0}^\infty \sum_{m=-\ell}^\ell \alpha_{\ell m}(t) Y_{\ell m}(\bx)$ follows from the general theory of Hilbert space valued Wiener process. It is true to say $W(t,\bx) = 
%\sum_{\ell=0}^\infty \sum_{m=-\ell}^\ell \alpha_{\ell m}(t) Y_{\ell m}(\bx)$ happens to be a KL expansion for fixed $t$'s, but I am not sure if it is true to say the KL expansion of the random field tells the representation if we have all the snapshots at all $t$'s, even though I bet that is how people tried to do when they generalized the Wiener process to the infinite dimension.}}, 
{The} $Q$-Wiener process taking values 
in {$H$} can be characterized by the Karhunen--Lo\`{e}ve expansion
\begin{equation}
W(t,\bx) = 
\sum_{\ell=0}^\infty \sum_{m=-\ell}^\ell \alpha_{\ell m}(t) Y_{\ell m}(\bx),
\end{equation}
where $\{\alpha_{\ell m}\}$ is given by
\begin{align}
\alpha_{\ell m}(t):=\sqrt{A}_{\ell} w_{\ell m}(t),
\end{align}
where $\{w_{\ell m}\}$ is a system of independent standard Brownian motions 
%$\{w_{\ell m}\}$ 
that are adapted to the underlying filtration with the usual condition. 
From this representation, we see that the corresponding $Q$-Wiener 
process satisfies the following: 
for any $t\in[0,1]$ the random field $W(t,\cdot)$ is an isotropic centered Gaussian random field:
\begin{align}
\bbE[W(t,\bx)W(t,\by)]= tK_T(\bx,\by).
\end{align}
%--------------------------------------------------------------------------------------------
\section{Stochastic evolution equations on the sphere}\label{sec:SDEs} 
\subsection{Existence and uniqueness} 
{In the following, $\alpha \preceq \beta$ means that $\alpha$ can be bounded by some constant times $\beta$ uniformly with respect to any parameters on which $\alpha$ and $\beta$ may depend. Further, $\alpha\asymp \beta$ means that $\alpha \preceq \beta$ and $\beta \preceq \alpha$. }

%{of the solution and its temporal regularity}
\begin{comment}
We defined the following Lebesgue-Bochner space
\[
L^2(\Omega;H)  := \{ X:  \|X\|^2_{L^2(\Omega;H)} = \bbE \|X \|^2_H < \infty  \}
\]
The covariance Cov$(X)$ is defined by
\[
 \text{Cov}(X) := \bbE[ (X-\bbE(X)) \otimes (X-\bbE(X))]
\]
{What does this $\otimes$ mean?}

The covariance Cov corresponds (1-1) to a unique
covariance operator $Q$
\[
  \langle\text{Cov}(X), \varphi \otimes \psi\rangle_{H \otimes H}
   = \langle Q \varphi, \psi \rangle_{H} 
\]
With the spherical harmonics $Y_{\ell m}$ to be the orthonormal
basis for $H = L^2(\mS^2)$, we can show that
\end{comment}
In order to define stochastic integrals with respect to the $Q$-Wiener process defined in the previous section, 
we introduce the Hilbert space
\[
  H_0 = Q^{1/2}(H),
\]
equipped with the inner product
\[
%\inprod{Q^{1/2} h_1}{Q^{1/2} h_2}_0 = \inprod{h_1}{h_2} \quad \mbox{ for } h_1, h_2 \in H.
\eqhl{\inprod{h_1}{h_2}_{H_0}=\sum_{\ell=0}^{\infty}\sum_{m=-\ell}^{\ell}
\YK{\frac1A_{\ell}}\inprod{h_1}{Y_{\ell m}}\inprod{h_2}{Y_{\ell m}},\quad\text{for $h_1,h_2\in H_0$}.}
\]
\YK{We note $\{\sqrt{A_{\ell}}Y_{\ell m}\}$ forms a complete orthonormal system in $H_0$.}
%In what follows, $L$ is the infinitesimal generator of a %strongly continuous semigroup
%$(S(t))_{t \ge 0}$ in $H$, $\xi \in H$, 

Let $(S(t))_{t \ge 0}$ be the strongly continuous operator semigroup acting on $H$ generated by $\Delta^{\ast}$. % {with the boundary condition $\int_{\mS^2}u(\bx)\mathrm{d}\varsigma(\bx)=0$}. 
Then, we have the spectral representation
\[
  S(t) u := \sum_{\ell=0}^{\infty}\sum_{m=-\ell}^{\ell} \exp(-\mu_\ell t) \inprod{u}{Y_{\ell m}} Y_{\ell m}\qquad\text{ for }\ u\in H.
\]
See, for example \cite{sell2013dynamics}.
 
%\footnote{
%{\tt
%@book{sell2013dynamics,
%  title={Dynamics of evolutionary equations},
%  author={Sell, George R and You, Yuncheng},
%  volume={143},
%  year={2013},
%  publisher={Springer Science \& Business Media}
%}
%}
%}
Let $\calL = \calL_2(H_0,H)$ {be the space of Hilbert--Schmidt operators from $H_0$ to $H$, 
and $\|\cdot\|_{\calL}$} denote the Hilbert--Schmidt norm. 
%The space $\calL$ is equipped with the $\sigma$-algebra that is generated by 
%all mappings of the from $A \mapsto \inprod{Af}{g}$ with $f\in H_0$ and $g \in H$. 
%\footnote{$\calL$ is a normed space with the usual Hilbert--Schmidt operator norms and we can use the Borel field corresponding to this norm topology, which I feel is very natural. Are these the same $\sigma$-algebras? If so, is this a standard way of writing this? If not, why do we need to use this $\sigma$-algebra? }Let
%$
%   B : [0,1] \times H \rightarrow \calL
%$
%be a measurable mapping.\footnote{Isn't the measurability implied by the Lipschitz condition?}
%We assume that $B$ is given by pointwise multiplication of the form
%\[
%  B(u) h = T_g(u)\cdot h, \text{ for } u\in H, \quad h \in H_0,
%\]
%where $T_g(x) = g \circ x$ with 
%$g \in C^1(\R)$ such that $\|g'\|_\infty<\infty$.
We assume that $B$ is Lipschitz {continuous in the following sense:}
\begin{equation}\label{equ:Lip}
   \YK{\| B(u) - B(v)\|_{\calL} \le C_{\mathrm{Lip}} \|u-v\|
   \quad \text{ for } \quad u,v\in H,}
\end{equation}
and $B$ satisfies the {following} linear growth condition:
\begin{align}\label{equ:lin growth}
 \YK{\|B(u)\|_{\calL} \le c (1+ \|u\|), 
   \quad \text{ for } \quad u\in H,}
\end{align}
for some positive constant $c$. 
In particular, \YK{$B\colon H \rightarrow \calL$ is $\calB(H)/\calB(\calL)$-measurable}.

In this work, we restrict our consideration to the operators $B$ of the 
form
\begin{align}
B(u)h = T_g(u)\cdot \widetilde{B}h, \text{ for } u\in H, \quad h \in H_0,
\label{equ:def B}
\end{align}
where {$T_g\colon H\to H$ is the Nemytskii operator}
\[
 { T_g(u)(\bx) := g(u(\bx)),\qquad \text{for }u\in H, \quad\bx\in\mS^2,}
\]
with $g \in C^1(\R)$ such that $\|g'\|_\infty:= \sup_{r\in\mathbb{R}}|g'(r)|< \infty$, {and $\widetilde{B} \colon \YK{H_0}\to H$ is given by}
\begin{align}
{
	\widetilde{B}h:=\sumlm\eta_{\ell m} \inprod{h}{Y_{\ell m}}Y_{\ell m},
}
\end{align}
with $\{\eta_{\ell m}\}\subset\mathbb{R}$ such that $\sup_{\ell,m}|\eta_{\ell m}|<\infty$.
We note that for $u\in H$ we indeed have $T_g(u)\in H$, as 
\begin{align}
\|T_g(u)\|
\le
  \int_{\mS^2} 
    \big( 2|g(0_{H}(\bx))|^2 + 2 \|g'\|_\infty^2 |u(\bx)|^2 \big)
  \mathrm{d}\varsigma(\bx) \preceq
  1+\|u\|.
\label{equ:Tgu H bd}
\end{align}
 This generalizes \cite{GroRit07a}, where $\widetilde{B}:=I$ was considered.
{We also note for $u,v\in H$ we have}
\begin{align}
{\|T_g(u)-T_g(v)\|\le\|g'\|_{\infty}\|u-v\|.\label{equ:Tgu H Lip}}
\end{align}
Such $B$ satisfies the aforementioned conditions: for $u\in H$ we have
\begin{align}
\| &B(u) \|_{\mathcal{L}(H_0,H)}^2
=
\sumlm
\int_{\mS^2} \big| g(u(\bx))  {\eta_{\ell m}}\sqrt{A_{\ell}}Y_{{\ell m}}(\bx)\big|^2 \mathrm{d}\varsigma(\bx)\\
&\le
{\Big(\sup_{\lambda,\nu}|\eta_{\lambda \nu}|\Big)^2}
	\sum_{ \ell = 0}^{ \infty } A_{\ell}
	\int_{\mS^2} 
		\big( 2|g(0_{H}(\bx))|^2 + 2 \|g'\|_\infty^2 |u(\bx)|^2 \big)
		\sum_{ m =-\ell }^\ell
		Y_{\ell m}(\bx)Y_{\ell m}(\bx)
	\mathrm{d}\varsigma(\bx) \\
&=
\frac{{\Big(\sup_{\lambda,\nu}|\eta_{\lambda \nu}|\Big)^2}}{ 4\pi } 
  \sum_{ \ell = 0}^{ \infty } A_{\ell}(2\ell + 1 )
    \big( 8\pi |g(0)|^2 + 2 \|g'\|_\infty^2 \|u\|_H^2 \big)<\infty.
\end{align}
Further, for $u,v\in H$ we have
\begin{align}
\| &B(u)-B(v) \|_{\mathcal{L}(H_0,H)}^2
%&=
%\sumlm
%\| T_g(u)\sqrt{A_{\ell}}Y_{{\ell m}} - T_g(v)\sqrt{A_{\ell}}Y_{{\ell m}}\|_{L^2}^2 \\
=
\sumlm
\int_{\mS^2} \big|\big(g(u(\bx))-g(v(\bx))\big) {\eta_{\ell m}}\sqrt{A_{\ell}}Y_{{\ell m}}(\bx)\big|^2 \mathrm{d}\varsigma(\bx)\\
&\le
{\Big(\sup_{\lambda,\nu}|\eta_{\lambda \nu}|\Big)^2}
\|g'\|_\infty 
	\sum_{ \ell = 0}^{ \infty } A_{\ell}
	\int_{\mS^2} 
		|u(\bx) - v(\bx) |^2
		\sum_{ m =-\ell }^\ell
		Y_{\ell m}(\bx)Y_{\ell m}(\bx)
	\mathrm{d}\varsigma(\bx) \\
&= 
{\Big(\sup_{\lambda,\nu}|\eta_{\lambda \nu}|\Big)^2}
\frac{\|g'\|_\infty}{{4\pi}}
	\sum_{ \ell = 0}^{ \infty } A_{\ell}{(2\ell+1)}
	\| u - v \|_H^2,
\end{align}
and thus the Lipschitz constant in \eqref{equ:Lip} is the square root of
\begin{align}
C_{\mathrm{Lip}}^2:=	\mathrm{Tr}Q\frac{\|g'\|_\infty}{{4\pi}}{\Big(\sup_{\lambda,\nu}|\eta_{\lambda \nu}|\Big)^2}.
\end{align}
Examples of $g$ are %${T_g}(u) = u$ 
{${T_g}(u)= a u + b$ for some given real numbers $a,b$}.
%\frac{1 + u^{p+1}}{1 + \|u\|^p}\;\; \text{for some } p>0.$$
%\footnote{{It seems unclear why this $g$ satisfies the above conditions. $g$ does not seem to be even a simple composition? Integration $\|\cdot\|$ needs to be calculated? Explanation or Reference?}}
We recall the following existence and uniqueness results for the solution from \cite[Section 7.1]{DaPZab14},
which is applicable {to} our problem.
\begin{theorem}%[Da Prato and Zabczyk 1992]
Under the assumptions that $B$ is Lipschitz and satisfies
the linear growth condition, there exists an continuous process 
$(X(t))_{t\in [0,1]}$ with values in $H$  which is adapted to the underlying filtration such that
\begin{equation}\label{equ:semigroup}
X(t) = S(t) \xi + \int_0^t S(t-s) \YK{B(X(s))}\mathrm{d}W(s), 
\quad t \in [0,1] \quad \text{$\bbP$-a.s.}
\end{equation}
Moreover, this process is uniquely determined $\bbP$-a.s., and it is called
the mild solution of the stochastic evolution equation
\[
 \mathrm{d}X(t) = \Delta^{\ast} X(t) \mathrm{d}t + \YK{B(X(t))} \mathrm{d}W(t), \quad X(0) = \xi.
\]
%{where $\Delta^{\ast}$ is the Laplace--Beltrami operator with the boundary condition $\int_{\mS^2}X(\bx)\mathrm{d}\varsigma(\bx)=0$.}
\end{theorem}
For $p \ge 1$, 
\begin{equation}\label{equ:Ebound}
\sup_{ t \in [0,T] } \bbE\|X(t)\|^p < \infty.
\end{equation}
%For more details, see Da Prato and Zabczyk \cite{DaPZab14}(Sec. 7.1).
Let
\begin{equation}\label{equ:defX}
X(t) = 
{\sum_{\ell=0}^{\infty}\sum_{m=-\ell}^{\ell}}
X_{\ell m}(t) Y_{\ell m}, 
  \quad X_{\ell m}(t) = 
  \inprod{X(t)}{Y_{\ell m}},
\end{equation}
for $\ell\in\N_{0}$. 
The processes $X_{\ell m} = (X_{\ell m}(t))_{t \in [0,1]}$
satisfy the following bi-infinite system of stochastic
differential equations
\begin{align*}
 \mathrm{d}X_{\ell m}(t) &= -\mu_\ell X_{\ell m}(t) \mathrm{d}t 
   + \sum_{\ell'=0}^\infty 
   \sum_{m'=-\ell'}^{\ell'}
   \sqrt{A_{\ell'}}\inprod{B(X(t)) Y_{\ell'm'}}{ Y_{\ell m}} 
   \mathrm{d} w_{\ell' m'}(t) \\
 X_{\ell m}(0) &= \inprod{\xi}{Y_{\ell m}} ,
 \qquad \ell \in \N_0, \; m=-\ell,\ldots,\ell.
\end{align*}
Each process $X_{\ell m}$ is given explicitly as
\begin{equation}\label{equ:defXlm}
   X_{\ell m}(t) = \exp(-\mu_\ell t) \inprod{\xi}{Y_{\ell m}}
   + 
   \sum_{\ell'=0}^\infty \sum_{m'=-\ell'}^{\ell'}  
   \sqrt{A_{\ell'}}
   Z_{\ell m,\ell' m'}(t),
\end{equation}
where 
\begin{equation}\label{equ:defZlm}
Z_{\ell m,\ell' m'}(t) = \int_{0}^t \exp(-\mu_\ell(t-s))
  \inprod{B(X(s)) Y_{\ell' m'}}{ Y_{\ell m}} \mathrm{d} w_{\ell' m'}(s).
\end{equation}
\YK{We note that the series $\sum_{\ell'=0}^\infty \sum_{m'=-\ell'}^{\ell'}  
   \sqrt{A_{\ell'}}
   Z_{\ell m,\ell' m'}(t)$ in the second term is convergent in $L^2(\Omega)$, due to \eqref{equ:lin growth} and \eqref{equ:Ebound}.}
\subsection{Temporal regularity}
The following regularity estimate will be used for the spatial truncation error estimate, 
see Theorem \ref{thm:semidiscrete}. A similar result for the stochastic PDE defined
on $[0,1]^d$ with {the Dirichlet condition} was proved in \cite{GroRit07}. In the same spirit, we prove
the following regularity result for the stochastic PDE defined on the unit sphere.
See also \cite[Theorem 9.1]{DaPZab14} for the mean-square continuity of the solution.
\begin{lemma}\label{lem:Xcts} 
%Suppose 
%If $B(u)=g(t,u)$ with \footnote{{The asumption is $B$ is Nemytskii, but do we need this here?}}
%\begin{align}
%\|g(t,u) - g(t,v)\| &\le c\|u - v\|  \\
%\|g(s,u) - g(t,u)\| &\le c |s-t|^{1/2}(1+ \|u\|)
%\end{align}
%\footnote{{Do we use the second condition?}}
%Then 
Suppose the Lipschitz condition \eqref{equ:Lip} and the linear growth condition \eqref{equ:lin growth} are satisfied. 
Then, the mild solution is continuous in the mean-square
sense on $[0,1]$. Further, we have the estimate
\[
  \bbE\|X(s)-X(t)\|^2 \le C|t-s|(1+ \psi(\min\{s,t\})),
\]
where $\psi \in L_1([0,1])$.
\end{lemma}

\begin{proof}
%\footnote{Comment: see Da Prato book Thm 9.1, also M\"{u}ller-Gronbach--Ritter (2007) FoCM p, 153}
\YK{First, note that we have}
\[
 \bbE \|X(s)-X(t)\|^2 = \sumlm \bbE(X_{\ell m}(s) - X_{\ell m}(t))^2.
\]
%\begin{equation}\label{equ:Xlm2}
%\begin{aligned}
%   X_{\ell m}(t) &= \exp(-\mu_\ell t) \inprod{\xi}{Y_{\ell m}} \\
%   &\qquad   + 
%   \sum_{\ell'=0}^\infty \sum_{m'=-\ell'}^{\ell'}  
%   \sqrt{A_{\ell'}}
%\int_{0}^t \exp(-\mu_\ell(t-r))
%  \inprod{B(X(r)) Y_{\ell' m'}}{ Y_{\ell m}} \mathrm{d} w_{\ell' m'}(r).
%\end{aligned}
%\end{equation}
\YK{For $s < t$, from \eqref{equ:defXlm} and \eqref{equ:defZlm} we have the identity}
\begin{align*}
X_{\ell m}(t) - X_{\ell m}(s) &= 
	\eqhl{[\exp(-\mu_\ell(t-s))-1]}X_{\ell m}(s) \\
&  + \sum_{\ell'=0}^\infty \sum_{m'=-\ell'}^{\ell'}
\int_{s}^{t}\exp(-\mu_\ell (t-r))
\sqrt{A_{\ell'}}\inprod{B(X(r)) Y_{\ell' m'}}{Y_{\ell m}}
\mathrm{d} w_{\ell' m'}(r),
\end{align*}

%\sout{where }
%\begin{empheq}[box=\shadebox]{align}
%a_{\ell',m'} = 
%\begin{cases}
%\sqrt{A_{\ell'}} \beta^{(1)}_{\ell',0}(t), & \text{ if  } m'=0 \\
%\sqrt{A_{\ell'}/2} (\beta^{(1)}_{\ell',m'}(t) + 
%i \beta^{(2)}_{\ell',m'}(t)) & \text{ if } m'  > 0,\\
%\sqrt{A_{\ell'}/2} (-1)^m({\beta^{(1)}_{\ell',-m'}(t)} - 
%i {\beta^{(2)}_{\ell',-m'}(t)}) & \text{ if } m'  < 0,
%\end{cases}
%\end{empheq}

%\sout{with $\beta^{(1)}(t)$ and $\beta^{(2)}(t)$ are $1$-dimensional real valued Wiener processes.}
By the It\^{o}'s isometry, we have
\begin{align*}
 \bbE(X_{\ell m}(s) - X_{\ell m}(t))^2 &=
 \eqhl{[\exp(-\mu_\ell (t-s))-1]^2}
 \bbE(X^2_{\ell m}(s)) \\
&+ \int_s^t \exp (-2\mu_\ell (t-r)) 
\bbE \| B^*(X(r)) Y_{\ell m} \|^2_{H_0} \mathrm{d}r,
\end{align*}
\YK{with}
\begin{comment}
\footnote{
{Comment: Let $\calB(H_0,H)$ be the vector space of bounded linear operators from the Hilbert spaces $H_0$ to the Hilbert space $H$. 
$B\in \calL(H_0,H)\implies 
B\in \calB(H_0,H) \implies
B^* \in \calB(H,H_0)\text{ and } (B^*)^*=B.
$ Therefore}
\begin{empheq}[box=box]{align*}
&\sum_{\ell'=0} \Big( A_{\ell'} |\inprod{B(X)Y_{\ell' 0}}{Y_{\ell m}}_{H}|^2+ 
    {A_{\ell'}} \sum_{|m'|\le \ell', m' \ne 0} 
    |\inprod{B(X) {Y_{\ell' m'}}} { {Y_{\ell m} }}_{H}|^2 \Big)\\
&
=
\sum_{\ell'=0} \Big(
		|\inprod{B(X) \sqrt{A_{\ell'}} Y_{\ell'0}}{ Y_{\ell m}}_{H}|^2+ 
		\sum_{|m'|\le \ell', m' \ne 0} 
    |\inprod{B(X) \sqrt{A_{\ell'}} Y_{\ell' m'}} { Y_{\ell m} }_{H}|^2 \Big) \\
&
=
\sum_{\ell'=0} \Big(
		|\inprod{\sqrt{A_{\ell'}} Y_{\ell' 0}}{ B^*(X)Y_{\ell m}}_{H_0}|^2+ 
		\sum_{|m'|\le \ell', m' \ne 0} 
    |\inprod{ \sqrt{A_{\ell'}} Y_{\ell' m'}} { B^*(X) Y_{\ell m} }_{H_0}|^2 \Big) \\
&=
\sum_{\ell'=0}^\infty
		\sum_{|m'|\le \ell'} 
    |\inprod{ B^*(X) Y_{\ell m} }{ \sqrt{A_{\ell'}} Y_{\ell' m'}} _{H_0}|^2
=
\|B^*(X) Y_{\ell,m}\|_{H_0}^2.
\end{empheq}
}% end footnote
\end{comment}
%
\begin{align}
\|B^*(X) Y_{\ell,m}\|_{H_0}^2 =
\sum_{\ell'=0} 
% A_{\ell'} 
%		|\inprod{B(X){Y_{\ell' 0}}}{{Y_{\ell m}}}|^2+ 
    A_{\ell'} \sum_{|m'|\le \ell'} 
    |\inprod{B(X) Y_{\ell' m'}}{ Y_{\ell m} }|^2\YK{,}
\end{align}
\YK{where $B^*(x):=(B(x))^*\colon H\to H_0$ for $x\in H$ denotes the adjoint operator of $B(x)$.} 
\begin{comment}
&\|B^*(X) Y_{\ell,m}\|_{H_0}^2 {=} \\
&\quad\sum_{\ell'=0} 
\left( A_{\ell'} 
		|\inprod{B(X){Y_{\ell' 0}}}{{Y_{\ell m}}}|^2+ 
    {A_{\ell'}} \sum_{|m'|\le \ell', m' \ne 0} 
    |\inprod{B(X) {Y_{\ell' m'}} } { {Y_{\ell m} }}|^2 \right).
\end{comment}
Similarly, we also have
\begin{equation}\label{equ:EXlm}
\begin{aligned}
 \bbE(X_{\ell m}(s))^2 &= \exp(-2 \mu_\ell s) \inprod{\xi}{Y_{\ell m}}^2 \\
 & + \int_0^s \exp(-2\mu_\ell (s-r)) 
 \bbE \| B^*(X(r)) Y_{\ell m}\|^2_{H_0} \mathrm{d}r.
\end{aligned}
\end{equation}
Put
\[
	\Gamma_1 = 
	\sumlm \eqhl{[\exp(-\mu_{\ell}(t-s))-1]^2} \bbE ( X_{\ell m}^2(s))
\]
and
\[
\Gamma_2 = \sumlm \int_s^t \exp(-2\mu_\ell(t-r)) 
\bbE \|B^*(X(r)) Y_{\ell m}\|^2_{H_0} \mathrm{d}r.
\]
We use \eqref{equ:Ebound} and the linear growth condition to obtain
\begin{equation}\label{equ:Gamma2}
\begin{aligned}
\Gamma_2 &\le \bbE\bigg( \int_{s}^t \sumlm 
 \|B^*(X(r)) Y_{\ell m}\|_{H_0}^2 \eqhl{\mathrm{d}r} \bigg)\\
 &=\bbE\bigg( \int_{s}^t  
 \|B^*(X(r))\|^2_{\calL(H,H_0)} \eqhl{\mathrm{d}r} \bigg) \le C(t-s).
\end{aligned}
\end{equation}
Fix $\varepsilon>0$ arbitrarily. Then, for sufficiently large $L_0$ we have
$$
\sup_{s\in[0,1]}\sum_{\ell=L_0+1}^{\infty}\sum_{m=-\ell}^{\ell}\bbE ( X_{\ell m}^2(s))<\frac{\varepsilon}{2}.
$$
Further, we can take $\delta>0$ such that for any $|t-s|<\delta$ and for any $1\le\ell\le L_0$ we have
$$
[\exp(-\mu_{\ell}(t-s))-1]^2 \bbE ( X_{\ell m}^2(s))<\frac{\varepsilon}{2L_0}.
$$ 
Thus, for such $s,t$ we have
$
\Gamma_1\le \varepsilon
$ , and thus together with \eqref{equ:Gamma2} the mean square continuity follows.

Now, since $1-\exp(-x) \YK{\le x}$ we have
$
\YK{\Gamma_1 \le \sum_{\ell=0}^{\infty} \mu_\ell(t-s) \bbE( X_{\ell m}^2(s))}
$, 
\YK{where the series is well defined since each term is non-negative.} 
Therefore, $\Gamma_1 \le (t-s) \psi(s)$ with 
\[
\psi(s):=\sumlm \mu_{\ell} \bbE(X^2_{\ell m}(s)).
\]
Since
\begin{align}\label{equ:muXlm}
\mu_{\ell} \int_0^1 \bbE(X^2_{\ell m}(s)) \mathrm{d}s 
\le \inprod{\xi}{Y_{\ell m}}^2
+ \int_0^1 \bbE \|B^*(X(r)) Y_{\ell m}\|^2_{H_0} \mathrm{d}r
\end{align}
we have 
\begin{equation}\label{equ:psiL1}
\psi \in L_1([0,1]).
\end{equation}
\end{proof}
\subsection{Isotropy of the solution}
The equation \eqref{equ:main} is driven by the Wiener process that is 2-weakly isotropic at each $t$. Thus, whether or not the isotropy propagates to the solution is of natural interest. In this section, we see that the solution does not necessarily have the isotropy in general. To see this, in this section we consider the Nemytskii operator $T_g$ with affine functions $g(x)=ax+b$, for some $a,b\in\mathbb{R}$.

We start from the following relation between the 2-weak isotropy and the eigenvalues of covariance operators.
\begin{proposition}
Let $Z=\{Z(\bx)\}_{\bx\in\mS^2}$ be a zero-mean random field on $\mS^2$ such that its  covariance function $K_{Z}(\cdot,\cdot)\colon \mS^2\times\mS^2\to\R$ is well-defined on all points $\mS^2\times\mS^2$, and $Z$ is $\mathcal{B}(\mS^2)\otimes\mathscr{F}/\mathcal{B}(\R)$-measurable. Then, $Z$ is $2$-weakly isotropic if and only if $K_{Z}\in L^2(\mS^2\times\mS^2)$ and the covariance operator $Q_Z$ defined as the integral operator $$H\ni h \mapsto Q_Z h:=\int_{\mS^2}K_{Z}(\cdot,\bx)h(\bx)\mathrm{d}\varsigma(\bx)\in H$$ 
has eigenfunctions $\{Y_{\ell m};{\ell\in\N_0,m=-\ell,\ldots,\ell\}}$ 
with eigenvalues independent of $m$.
\end{proposition}
\begin{proof}
Since $Z$ is $2$-weakly isotropic, letting $\bx_{\mathrm{n}}$ be the north pole we have $K_{Z}\in L^2(\mS^2\times\mS^2)$:
\begin{align}
\int_{\mS^2}\!\!\int_{\mS^2}|K_{Z}(\bx_1,\bx_2)|^2\mathrm{d}\varsigma(\bx_1)\mathrm{d}\varsigma(\bx_2)
&\le 
\int_{\mS^2}\!\!\int_{\mS^2}
|\bbE[Z(\bx_1)^2]||\bbE[Z(\bx_2)^2]|
\mathrm{d}\varsigma(\bx_1)\mathrm{d}\varsigma(\bx_2)\\
&=
16\pi^2
|\bbE[Z(\bx_{\mathrm{n}})^2]|^2<\infty.
\end{align}
Thus we have the expansion of $K_{Z}$ as 
\begin{align}
K_{Z}(\bx_1,\bx_2)=\sum_{\ell=0}^\infty \frac{2\ell+1}{4\pi}A_{\ell}P_{\ell}(\bx_1\cdot\bx_2)
=\sumlm A_{\ell} Y_{\ell m}(\bx_1)Y_{\ell m}(\bx_2),
\end{align}
with some sequence $\{A_{\ell}\}$, where in the second equality the addition theorem is used. Thus, we immediately see $Q_Z$ has the eigenpair $(A_{\ell},Y_{\ell m})$.
Conversely, if we have $K_{Z}\in L^2(\mS^2\times\mS^2)$, then we have the representation
$$
K_{Z}(\bx_1,\bx_2)
=\sum_{\lambda=0}^{\infty}\sum_{\nu=-\lambda}^{\lambda}
\sum_{\lambda'=0}^{\infty}\sum_{\nu'=-\lambda'}^{\lambda'}
C_{\lambda \nu \lambda'\nu'} Y_{\lambda \nu}(\bx_1) Y_{\lambda' \nu'}(\bx_2),
$$
for some $\{C_{\ell m \ell'm'}\}$ in $L^2(\mS^2\times\mS^2)$. From the assumption we have
$$
Q_Z Y_{\ell m}=
\int_{\mS^2}K_{Z}(\cdot,\bx)Y_{\ell m}(\bx)\mathrm{d}\varsigma(\bx)=A_{\ell} Y_{\ell m},$$
thus
$
\sum_{\lambda=0}^{\infty}\sum_{\nu=-\lambda}^{\lambda}
C_{\lambda \nu \ell m} Y_{\lambda \nu}=A_{\ell} Y_{\ell m}
$. Multiplying $\{Y_{\ell m}\}$ to both sides and integrating over $\mS^2$ yields $C_{\ell m \ell m} =A_{\ell}$ and $C_{\lambda \nu \ell m} =0$ unless $\lambda=\ell$ and $\nu=m$. 
\end{proof}
From the previous result, we expect that the isotropy should be preserved as long as the isotropic noise $W$ is acted {upon} by operators that are diagonalised by $\{Y_{\ell,m}\}$ 
%such that 
{and} their eigenvalues do not depend on $m$'s. Now, we note that the mapping $B$ is defined by 
the pointwise multiplication and a Nemytskij operator. This makes the analysis difficult, 
because non-trivial multiplication operators cannot be diagonalised by $\{Y_{\ell,m}\}$ as 
we see in the next proposition.

Let $f\in H$ {and} the multiplication operator $\mulop_{f}\colon H_{0}\to H$ {be} defined by
$$
(\mulop_{f}h)(\bx)=f(\bx)h(\bx)\quad \text{ for }h\in H_{0}.
$$
%satisfies $\mulop_{f}\in \mathcal{L}$.
\begin{proposition}
Suppose $f\in {H}$ defines a multiplication operator {such that} $\mulop_{f}\in \mathcal{L}$. Suppose {further} that $f$ is not a constant function on $\mS^2$. 
Then, $\mulop_{f}$ cannot be diagonalised by $\{Y_{\ell m}\}$. In particular, 
$\mulop_{f}$ cannot not have $\{Y_{\ell m}\}$ as eigenfunctions.
\end{proposition}
\begin{proof}
\begin{sloppypar}
Suppose the $\mulop_{f}$ can be diagonalised by $Y_{\ell m}$, i.e., for $h\in H$ we have
$
\mulop_{f}h = fh= \sumlm c_{\ell m}\inprod{h}{Y_{\ell m}}Y_{\ell m}
$ with some $\{c_{\ell m}\}\subset \mathbb{R}$. In particular, 
 we have 
$
(\mulop_{f}h)(\bx)=f(\bx) Y_{\lambda\nu}(\bx)=c_{\lambda \nu} Y_{\lambda\nu}(\bx).
$ 
For all $\lambda>1$, $\nu=-\lambda,\dots,\lambda$, integrating both sides over $\mS^2$ yields 
\begin{align}
\int_{\mS^2}f(\bx) Y_{\lambda\nu}(\bx)\mathrm{d}\varsigma(\bx)=0.
\end{align}
\end{sloppypar}
\noindent Thus, we must have $f\in \mathrm{span}\{Y_{00}\}$, which contradicts the assumption.
\end{proof}
\subsubsection{Equations for the second moment}
To study the propagation of 2-weak isotropy, we formulate the equations for the second moment---the equation that has the covariance function of the solution $X$ as the solution---as an abstract Cauchy problem in $L^2(\mS^2\times\mS^2)$.

In the following we assume the initial condition $\xi$ is constant over $\mS^2$ so that the deterministic random field $\xi$ is 2-weakly isotropic.

Let $S\diamond BW:=\int_0^{\cdot} S(t-s) {B}(X(s))\mathrm{d}W(s)$. Then, 
$
X(t)
=
S(t)\xi+
(S\diamond BW)(t)
\in H$, and $\bbE[S\diamond BW]=0$. Since $\bbE[X(t)]=S(t)\xi$, we see that $\mS^2\ni\bx\mapsto\bbE[X(t,\bx)]$ is a constant function given that $\xi$ is constant over $\mS^2$. Thus, to show the 2-weak isotropy of $X(t,\cdot)$, it suffices to {show that for any fixed $t>0$} the covariance of $X(t,\bx_1)$ and $X(t,\bx_2)$ is rotationally invariant for any $\bx_1,\bx_2\in\mS^2$.

We start with the following formula.
\begin{lemma}
For any $\ell,\ell'\ge0$, {$|m| \le \ell$, $|m'|\le \ell'$}, we have 
\begin{align}
&\mathbb{E}\Big[
	\big\langle
	X(t)-S(t)\xi
	,Y_{\ell m}
	\big\rangle 
	\big\langle
	X(t)-S(t)\xi
	,Y_{\ell' m'}
	\big\rangle 
\Big]\nonumber\\
&=
\sum_{\lambda=0}^{\infty} \sum_{\nu=-\lambda}^{\lambda}
\mathbb{E}\bigg[
\int_{0}^t
\inprod{\sqrt{A}_{\lambda} S(t-s) B(X(s))Y_{\lambda \nu}}{Y_{\ell m}}
\inprod{\sqrt{A}_{\lambda} S(t-s) B(X(s))Y_{\lambda \nu}}{Y_{\ell' m'}}
\mathrm{d}s\bigg].\label{eq:E fourier coeff expansion}
\end{align}
\end{lemma}
\begin{proof}
Since $X(t)=S(t)\xi+\int_0^{t} S(t-s) {B}(X(s))\mathrm{d}W(s)$, 
\begin{comment}
\mathbb{E}\Bigg[
	\Big\langle
	X(t)
	,Y_{\ell' m'}
	\Big\rangle 
	\Big\langle
	X(t)
	,Y_{\ell m}
	\Big\rangle 
\Bigg]
&=\mathbb{E}\Bigg[
	\Big\langle
	(S\diamond BW)(t)
	,Y_{\ell' m'}
	\Big\rangle 
	\Big\langle
	(S\diamond BW)(t)
	,Y_{\ell m}
	\Big\rangle 
\Bigg].
\end{comment}
we have that for any $h\in H$ 
\begin{align}
\big\langle
	X(t)-S(t)\xi
	,h
	\big\rangle 
=
\sum_{\lambda=0}^{\infty} \sum_{\nu=-\lambda}^{\lambda}
\int_{0}^t\inprod{\sqrt{A}_{\lambda} S(t-s) B(X(s))Y_{\lambda \nu}}{h}\mathrm{d}w_{\lambda \nu}(s),
\end{align}
with the series convergent in $L^{2}(\Omega)$.

Since $\{w_{\lambda \nu}\}$ are independent standard Brownian motions and 
thus their quadratic % {covariance} 
{covariations}
vanish  unless the indices $\lambda$ and $\nu$ coincide, the It\^{o}'s isometry implies
\begin{align}
&\mathbb{E}\bigg[
\int_{0}^t
\inprod{\sqrt{A}_{\lambda}S(t-s) B(X(s))Y_{\lambda \nu}}{h}\mathrm{d}w_{\lambda \nu}(s)
\int_{0}^t
\inprod{\sqrt{A}_{\lambda'}S(t-s) B(X(s))Y_{\lambda' \nu'}}{h'}\mathrm{d}w_{\lambda' \nu'}(s)\bigg]\nonumber\\
&
=
\delta_{\lambda\lambda'}\delta_{\nu\nu'}
\mathbb{E}\bigg[
\int_{0}^t
\inprod{\sqrt{A}_{\lambda}S(t-s) B(X(s))Y_{\lambda \nu}}{h}
\inprod{\sqrt{A}_{\lambda}S(t-s) B(X(s))Y_{\lambda \nu}}{h'}\mathrm{d}s\bigg].
\end{align}
From these two facts we have
\begin{align*}
&\mathbb{E}\Big[
	\big\langle
	X(t)-S(t)\xi
	,Y_{\ell m}
	\big\rangle 
	\big\langle
	X(t)-S(t)\xi
	,Y_{\ell' m'}
	\big\rangle 
\Big]\nonumber\\
&=
\sum_{\lambda=0}^{\infty} \sum_{\nu=-\lambda}^{\lambda}
\mathbb{E}\bigg[
\int_{0}^t
\inprod{\sqrt{A}_{\lambda} S(t-s) B(X(s))Y_{\lambda \nu}}{Y_{\ell m}}
\inprod{\sqrt{A}_{\lambda} S(t-s) B(X(s))Y_{\lambda \nu}}{Y_{\ell' m'}}
\mathrm{d}s\bigg],
\end{align*}
which completes the proof.
\end{proof}
%The above lemma does not use the specific form of the function $g(x)=x$. 
Now we assume  $g(x)=x$ for $x\in\mathbb{R}$. That is,
 $(B(u)h)(\bx)=u(\bx)(\widetilde{B}h)(\bx)$. 
Noting that $\sup_{s\in[0,1]}\|X(s)\|^2<\infty$, we have 
\begin{align}
&\mathbb{E}\Big[
	\big\langle
	X(t)-S(t)\xi
	,Y_{\ell m}
	\big\rangle 
	\big\langle
	X(t)-S(t)\xi
	,Y_{\ell' m'}
	\big\rangle 
\Big]\nonumber\\
&=\mathbb{E}\Big[
	\big\langle
	X(t)
	,Y_{\ell m}
	\big\rangle 
	\big\langle
	X(t)
	,Y_{\ell' m'}
	\big\rangle\Big]-
	\big\langle S(t)\xi,Y_{\ell m}\big\rangle
	\big\langle S(t)\xi,Y_{\ell' m'}\big\rangle
\nonumber\\
&=
\int_{\mS^2}\!\!\int_{\mS^2}
\mathbb{E}
\int_0^t
[X(t,\bx_1)X(t,\bx_2)]Y_{\ell m}(\bx_1)Y_{\ell' m'}(\bx_2)
\mathrm{d}s
\mathrm{d}\varsigma(\bx_1)\mathrm{d}\varsigma(\bx_2)\nonumber\\
&\phantom{==}-
\int_{\mS^2}\!\!\int_{\mS^2} S(t)\xi(\bx_1)\cdot S(t)\xi(\bx_2)
Y_{\ell m}(\bx_1)Y_{\ell' m'}(\bx_2) \mathrm{d}\varsigma(\bx_1)\mathrm{d}\varsigma(\bx_2),
\end{align}
where in the first equality we used that 
\YK{$\bbE\langle X(t),Y_{\ell' m'}\rangle=\langle S(t)\xi,Y_{\ell' m'}\rangle
$}.

Further, since 
$
\big\langle{\sqrt{A}_{\lambda} S(t-s) B(X(s))Y_{\lambda \nu}},{Y_{\ell m}}\big\rangle
%&{=
%\big\langle{\sqrt{A}_{\lambda} \widetilde{B}Y_{\lambda \nu}}, S(t-s)^*{Y_{\ell m}X(s)}\big\rangle}\\&
=
\big\langle{\sqrt{A}_{\lambda} X(s) \mathrm{e}^{-\mu_{\lambda}(t-s)} \eta_{\lambda\nu}Y_{\lambda \nu}}, {Y_{\ell m}}\big\rangle
$, 
we have
\begin{align*}
\big\langle{|\sqrt{A}_{\lambda} S(t-s) B(X(s))Y_{\lambda \nu}|},{|Y_{\ell m}|}\big\rangle
\le 
c\sqrt{2\ell+1}\sup_{s\in[0,1]}\|X(s)\|\sqrt{A}_{\lambda}
{\mathrm{e}^{-\mu_{\lambda}(t-s)}}{\sqrt{2\lambda+1}}
\eta_{\lambda\nu}<\infty.
\end{align*}
\begin{sloppypar}
\noindent Noting
$\sum_{\lambda=0}^\infty \sum_{\nu=-\lambda}^{\lambda}{A}_{\lambda}
(2\lambda+1)\int_{0}^t {\mathrm{e}^{-2\mu_{\lambda}(t-s)}}\mathrm{d}s
\le \sum_{\lambda=0}^\infty \sum_{\nu=-\lambda}^{\lambda}{A}_{\lambda}
\frac{2\lambda+1}{\mu_{\lambda}}<\infty
$, together with $
\sup_{\lambda,\nu}|\eta_{\lambda\nu}|<\infty$ 
we can rewrite \eqref{eq:E fourier coeff expansion} by changing the order of the integrals as
\end{sloppypar}
\begin{align}
&\int_{\mS^2}\!\!\int_{\mS^2}
\mathbb{E}
[X(t,\bx_1)X(t,\bx_2)]Y_{\ell m}(\bx_1)Y_{\ell' m'}(\bx_2)
\mathrm{d}\varsigma(\bx_1)\mathrm{d}\varsigma(\bx_2)\nonumber \\
&\phantom{===============}-
\int_{\mS^2}\!\!\int_{\mS^2} S(t)\xi(\bx_1)\cdot S(t)\xi(\bx_2)
Y_{\ell m}(\bx_1)Y_{\ell' m'}(\bx_2)
\mathrm{d}\varsigma(\bx_1)\mathrm{d}\varsigma(\bx_2)\\
&=
\int_{\mS^2}\!\!\int_{\mS^2}
\int_{0}^t
\sum_{\lambda=0}^\infty\sum_{\nu=-\lambda}^{\lambda}
\!{A}_{\lambda}
\mathbb{E}[X(s,\bx_1)X(s,\bx_2)] \mathrm{e}^{-2\mu_{\lambda}(t-s)} \eta_{\lambda\nu}^2
Y_{\lambda \nu}(\bx_1)Y_{\lambda \nu}(\bx_2)
\mathrm{d}s\nonumber\\
&\phantom{===============}\times
Y_{\ell m}(\bx_1)
Y_{\ell' m'}(\bx_2)
\mathrm{d}\varsigma(\bx_1)\mathrm{d}\varsigma(\bx_2).
\label{eq:sol is weak}
\end{align}
This identity suggests that the kernel $(\bx_1,\bx_2)\mapsto\mathbb{E}[X(s,\bx_1)X(s,\bx_2)]$ is the weak solution of an abstract Cauchy problem in $L^2(\mS^2\times\mS^2)$. 
Then, the rotational invariance of the covariance is nothing but this function being a zonal kernel. 
This motivates us to study {an abstract} Cauchy problem in the space of zonal kernels.

The operator $F$ defined below will be used as the forcing term for the abstract Cauchy problem we consider.
\begin{lemma}
Let \eqref{equ:Aell summable} be satisfied and let
\begin{align}
\kappa(\bx_1,\bx_2):=\sum_{\lambda=0}^\infty\sum_{\nu=-\lambda}^{\lambda}\!{A}_{\lambda}
\eta_{\lambda\nu}^2
Y_{\lambda \nu}(\bx_1)Y_{\lambda \nu}(\bx_2),\quad\text{ for }\bx_1,\bx_2\in \mS^2.
\end{align}
Then, the multiplication operator $F=F_{\kappa}$ on $L^{2}(\mS^2\times\mS^2)$ defined by 
\begin{align}
v(\bx_1,\bx_2)\mapsto \kappa(\bx_1,\bx_2)v(\bx_1,\bx_2)=:Fv(\bx_1,\bx_2),\quad (\bx_1,\bx_2)\in\mS^2\times\mS^2,
\label{eq:def F}
\end{align}
is bounded as an operator from $L^2(\mS^2\times\mS^2)$ to $L^2(\mS^2\times\mS^2)$.
\end{lemma}
\begin{proof}
First, note that from the Cauchy--Schwarz inequality and the addition theorem, 
{the condition} 
\eqref{equ:Aell summable} implies %\footnote{{Could you check (3.34) to (3.36)?}}
\begin{align}
&\sup_{(\bx_1,\bx_2)\in\mS^2\times\mS^2}|\kappa(\bx_1,\bx_2)|
\le
\sum_{\lambda=0}^\infty
{A}_{\lambda}
\sum_{\nu=-\lambda}^{\lambda}\!\Big|
\eta_{\lambda\nu}^2
Y_{\lambda \nu}(\bx_1)Y_{\lambda \nu}(\bx_2)\Big|\\
&
\le 
\sum_{\lambda=0}^\infty {A}_{\lambda}
\bigg(\sum_{\nu=-\lambda}^{\lambda}\!
\eta_{\lambda\nu}^2
Y_{\lambda \nu}(\bx_1)^2\bigg)^{\frac12}
\bigg(\sum_{\nu=-\lambda}^{\lambda}\!
\eta_{\lambda\nu}^2
Y_{\lambda \nu}(\bx_2)^2\bigg)^{\frac12}
\le 
\big(\sup_{\lambda',\nu'}\eta_{\lambda'\nu'}^2\big)
\sum_{\lambda=0}^\infty {A}_{\lambda}\frac{2\lambda+1}{4\pi}<\infty.
\end{align}
Thus, we have
\begin{align}
\int_{\mS^2}\!\int_{\mS^2}
|Fv(\bx_1,\bx_2)|^2
\mathrm{d}\varsigma(\bx_1)\mathrm{d}\varsigma(\bx_2)
\le 
\big(\sup_{(\bx_1,\bx_2)\in\mS^2\times\mS^2}|\kappa(\bx_1,\bx_2)|^2\big)
\|v\|_{L^2(\mS^2\times\mS^2)}^2<\infty.
\end{align}
\end{proof}
In the following, we abuse the notation slightly by writing 
$$(Y^1_{\lambda\nu}Y^2_{\lambda'\nu'})(\bx_1,\bx_2):=Y_{\lambda\nu}(\bx_1)Y_{\lambda'\nu'}(\bx_2).$$ Note that $\{Y^1_{\lambda\nu}Y^2_{\lambda'\nu'}\}$ is a complete orthonormal system for $L^2(\mS^2\times\mS^2)$. 
Now, {we} define the operator $\blacktriangle\colon L^{2}(\mS^2\times\mS^2)\to L^{2}(\mS^2\times\mS^2)$ by 
$
\blacktriangle Y^1_{\lambda\nu}Y^2_{\lambda'\nu'}:=
-(\mu_{\lambda}+\mu_{\lambda'})Y^1_{\lambda\nu}Y^2_{\lambda'\nu'}
$. {Then, } $\blacktriangle$ is self-adjoint with the domain
\begin{align}
&D(\blacktriangle)=D(\blacktriangle^*)=
\Big\{f\in L^2(\mS^2\times\mS^2)\colon \sum_{\lambda,\lambda'}\sum_{\nu,\nu'}
(\mu_{\lambda}+\mu_{\lambda'})^2\
\langle f,Y^1_{\lambda \nu}Y^2_{\lambda' \nu'}\rangle_{L^2(\mS^2\times\mS^2)}^2<\infty
\Big\},
\end{align}
which is densely defined in $L^2(\mS^2\times\mS^2)$. 
We note that $-\blacktriangle$ is positive. Thus, $\blacktriangle$ generates a $C_0$-semigroup on $L^2(\mS^2\times\mS^2)$. 
Thus, %the solution of 
the initial value problem
\begin{align}
\frac{\mathrm{d}v}{\mathrm{d}t}(t)=\blacktriangle v(t)+F(v(t)),\quad 
v(0)=v_0\in L^2(\mS^2\times\mS^2)
\end{align}
\begin{sloppypar}
\noindent has the unique mild solution 
$
v(t)=\mathrm{e}^{t\blacktriangle}v_0+\int_0^t \mathrm{e}^{(t-s)\blacktriangle}F(v(s))\mathrm{d}s\in L^2(\mS^2\times\mS^2)$. We note that $(\mathrm{e}^{\blacktriangle t}h^1h^2)(\bx_1,\bx_2)=(S(t)h)(\bx_1)(S(t)h)(\bx_2)$, where $(h^1h^2)(\bx_1,\bx_2):=h(\bx_1)h(\bx_2)$ for $h\in H$.
\end{sloppypar}
We now let 
$$
V_{O}:=\{ f\in L^2(\mS^2\times\mS^2) \colon f(\bx_1,\bx_2)=f(O\bx_1,O\bx_2)\,\text{for any}\, O\in SO(3) \}.
$$
denote the space of zonal functions. 
\begin{proposition}\label{prop:indep then zonal}
Let $F\colon L^2(\mS^2\times\mS^2)\to L^2(\mS^2\times\mS^2)$ be defined as in \eqref{eq:def F}. Suppose $\eta_{\lambda \nu}$ is independent of $\nu$, i.e.,  $\eta_{\lambda \nu}={\eta}_{\lambda}$
{ for all $\lambda \in\N_0$, $|\nu|\le\lambda$}. 
Then, the initial value problem
\begin{align}\label{eq:IVP}
\begin{cases}
\frac{\mathrm{d}v}{\mathrm{d}t}(t)=\blacktriangle v(t)+F(v(t)),\quad 
v(0)=v_0\in V_O
\end{cases}
\end{align}
has the unique mild solution
$
v(t)=\mathrm{e}^{t\blacktriangle}v_0+\int_0^t \mathrm{e}^{(t-s)\blacktriangle}F(v(s))\mathrm{d}s\in V_O
$ in the space of zonal kernels. 
%In particular, with $v(0)=0$ we have 
%$v(t)=\int_0^t \mathrm{e}^{(t-s)\blacktriangle}F(v(s))\mathrm{d}s\in V_O$.
\end{proposition}
\begin{proof}
Noting that any zonal function in $L^2(\mS^2\times\mS^2)$ can be expanded 
{ by the Legendre polynomials with a unique $\ell^2$-expansion coefficients, }
%{as a sequence of} Legendre polynomials 
we observe that $V_{O}$ is a closed %\footnote{{Is this clear enough?}}
%\footnote{Any zonal function can be expanded by the Legendre polynomial. Let $\{f_n\}\subset V_{O}$ a sequence convergent in $L^2(\mS^2\times\mS^2)$. Then $\{f_n\}$ is Cauchy in $L^2(\mS^2\times\mS^2)$. Show that the corresponding sequence defined by the expansion coefficients by Legendre polynomials is Cauchy in $\ell^2$, and thus convergent. Take the limit in $\ell^2$, and consider the zonal kernel $f^*$ defined by the coefficients, which is clearly in $V_{O}$. Show that $f^*$ is the limit of $\{f_n\}$ in $L^2(\mS^2\times\mS^2)$.} 
subspace in $L^2(\mS^2\times\mS^2)$. Thus, $V_O$ itself is a Hilbert space.
\begin{sloppypar}
Further, $Fv=\kappa\cdot v\in V_{O}$ for $v\in V_{O}$, since $\eta_{\lambda\nu}=\eta_{\lambda}$. 
Finally, we claim $\blacktriangle\colon V_{O}\cap D(\blacktriangle)\to  V_{O}$, and $V_{O}\cap D(\blacktriangle)$ is dense in $V_{O}$.
Indeed, since $v\in V_{O}$ is zonal, we have the representation
$
v(\bx_1,\bx_2)=
\sum_{\lambda=0}^\infty \sum_{\nu=-\lambda}^\lambda  C_{\lambda}Y^1_{\lambda \nu}Y^2_{\lambda \nu}
({\bx_1,\bx_2})
$ in $L^2(\mS^2\times\mS^2)$ for some {sequence} $\{C_{\lambda}\}$ 
{in $\ell^2$}. Thus, 
\begin{align}
\blacktriangle v
&=
\sum_{\ell,\ell'=0}^{\infty}\sum_{m,m'}
-(\mu_\ell+\mu_{\ell'})\inprod{v}{Y^1_{\ell m}Y^2_{\ell' m'}}_{L^2(\mS^2\times\mS^2)}
Y^1_{\ell m}Y^2_{\ell' m'}\\
&=
\sum_{\lambda'=0}^{\infty}\sum_{\nu'=-\lambda'}^{\lambda'}
-2\mu_{\lambda'}\inprod{v}{Y^1_{\lambda' \nu'}Y^2_{\lambda' \nu'}}_{L^2(\mS^2\times\mS^2)}
Y^1_{\lambda' \nu'}Y^2_{\lambda' \nu'}
=
\sum_{\lambda'=0}^{\infty}\sum_{\nu'=-\lambda'}^{\lambda'}
-2\mu_{\lambda'} C_{\lambda}
Y^1_{\lambda' \nu'}Y^2_{\lambda' \nu'},
\end{align}
which is zonal and thus in $V_{O}$.
Further, for any $v\in V_{O}$ the truncation $v^N(\bx_1,\bx_2)=
\sum_{\lambda=0}^N \sum_{\nu=-\lambda}^\lambda  C_{\lambda}Y^1_{\lambda \nu}Y^2_{\lambda \nu}({\bx_1,\bx_2})$ is in $V_{O}\cap D(\blacktriangle)$, but since $v^N$
is convergent in $L^2(\mS^2\times\mS^2)$ we have $\sum_{\lambda=0}^{\infty}\sum_{\nu=-\lambda}^\lambda C_{\lambda}^2<\infty$. From  $\|v-v^N\|_{L^2(\mS^2\times\mS^2)}=\sum_{\lambda>N}C_{\lambda}^2$ we have the density.
\end{sloppypar}
{Hence, we can conclude that} \eqref{eq:IVP} {is an initial value problem on the Hilbert space $V_{O}$, and hence $v$ develops in $V_{O}$.}
\end{proof}
We have the converse. 
\begin{proposition}\label{prop:zonal then indep}
Let $F\colon L^2(\mS^2\times\mS^2)\to L^2(\mS^2\times\mS^2)$ be defined as in \eqref{eq:def F}.  Suppose that the initial value problem
\begin{align}
\begin{cases}
\frac{\mathrm{d}v}{\mathrm{d}t}(t)=\blacktriangle v(t)+F(v(t)),\quad 
v(0)=v_0\in V_O
\end{cases}
\end{align}
has the unique mild solution
$
v(t)=\mathrm{e}^{t\blacktriangle}v_0+\int_0^t \mathrm{e}^{(t-s)\blacktriangle}F(v(s))\mathrm{d}s\in V_O
$ in the space of zonal kernels. 
Then, $\eta_{\lambda\nu}$ must be independent of $\nu$ 
for all $\lambda \in \N_0$, $|\nu|\le\lambda$.
\end{proposition}
\begin{proof}
%\footnote{{Is this proof clear?}}
We show that if $\eta_{\lambda\nu}$ depends on $\nu$, then $v(t)\not\in V_{O}$. 
First, consider the case where there exists one $\lambda^*\in\mathbb{N}_0$ such that $\eta_{\lambda^*\nu^*}$ depends on $\nu^*\in\{-\lambda^*,\dots,\lambda^*\}$.

 Consider the multiplication operator $F_{\lambda^*}\colon L^2(\mS^2\times\mS^2)\to L^2(\mS^2\times\mS^2)$ defined by 
$$
F_{\lambda^*}v(\bx_1,\bx_2):= {v(\bx_1,\bx_2)\kappa_{\lambda^*}(\bx_1,\bx_2)},
$$
{with $\kappa_{\lambda^*}(\bx_1,\bx_2):={A}_{\lambda^*}\sum_{\nu^*=-\lambda^*}^{\lambda^*}
\eta_{\lambda^*\nu^*}^2
Y_{\lambda^* \nu^*}(\bx_1)Y_{\lambda^* \nu^*}(\bx_2)$.} 
We claim that for $v\in V_{O}$ we must have $F_{\lambda^*}v\not\in V_{O}$.
{To see this, it suffices to show $\kappa_{\lambda^*}(\bx_1,\bx_2)\not\in V_{O}$. }
Suppose otherwise. Then, with some $\{C_{\lambda}\}$ we have the representation
$$
{\kappa_{\lambda^*}(\bx_1,\bx_2)}=
\sum_{\lambda=0}^{\infty} \sum_{\nu=-\lambda}^\lambda  C_{\lambda}Y^1_{\lambda \nu}Y^2_{\lambda \nu}
({\bx_1,\bx_2}).
$$
Multiplying $Y_{\lambda^* \nu^*}(\bx_1)Y_{\lambda^* \nu^*}(\bx_2)$ to both sides and integrating implies $\eta_{\lambda^*\nu^*}$ is independent of $\nu^*$, contradiction. Hence we have $F_{\lambda^*}v\not\in V_{O}$.

It suffices to consider the case where there exists one $\lambda^*\in\mathbb{N}_0$ such that $\eta_{\lambda^*\nu^*}$ depends on $\nu^*\in\{-\lambda^*,\dots,\lambda^*\}$. This is because zonal kernels cannot be expressed by a sum of non-zonal kernels. 
%Indeed, if we equate a zonal kernel with a sum of non-zonal kernels. Then multiply both sides with $Y^1_{\ell m}Y^2_{\ell' m'}$, and integrate over $\mS^2\times\mS^2$. Then, when $\ell\neq\ell'$ or $m\neq m'$, the both sides must be zero, which forces the expansions to be the Legendre polynomial expansion. Thus, all the kernels must be zonal.

Hence, we conclude if $\eta_{\lambda\nu}$ depends on $\nu$ then  $v(t)={\mathrm{e}^{t\blacktriangle}v_0+\int_0^t \mathrm{e}^{(t-s)\blacktriangle}F(v(s))\mathrm{d}s}\not\in V_O$.
\end{proof}
Now we go back to the stochastic heat equation, and characterize the
{2-weak isotropy of the solution}. 
% 2-weakly isotropy {property of the solution}. 
\begin{proposition}
\begin{sloppypar}
Suppose the operator $B$ is defined by $(B(u)h)(\bx)=u(\bx)(\widetilde{B}h)(\bx)$ with $\widetilde{B}h=\sumlm\eta_{\ell m}\inprod{h}{Y_{\ell m}}Y_{\ell m}$. Then, the solution the stochastic heat equation $X$ with an initial condition $\xi\in H$ that is constant over $\mS^2$ is 2-weakly isotropic if and only if $\eta_{\ell m}$ is independent of $m$, i.e., $\eta_{\ell m}=\eta'_{\ell}$ with some $\eta'_{\ell}$
{for all $\ell \in \N_0$, $|m| \le \ell$}.
\end{sloppypar}
\end{proposition}
\begin{proof}
The mild solution $v$ of the problem \eqref{eq:IVP} with $v(0)={\xi^1\xi^2}$ satisfies 
the integral equation of the form \eqref{eq:sol is weak}:
\begin{align}
\inprod{v(t)-{\mathrm{e}^{t\blacktriangle}\xi^1\xi^2}}{Y^1_{\ell m}Y^2_{\ell' m'}}_{L^2(\mS^2\times\mS^2)}
=
\inprod{\int_0^t \mathrm{e}^{(t-s)\blacktriangle}F(v(s))\mathrm{d}s}{Y^1_{\ell m}Y^2_{\ell' m'}}_{L^2(\mS^2\times\mS^2)}.
\end{align}
Thus, letting $w(\bx_1,\bx_2):=\mathbb{E}[X(s,\bx_1)X(s,\bx_2)]$, for any $(\ell,m,\ell',m')$ 
{where $\ell,\ell' \in \N_0; |m|\le \ell, |m'|\le \ell'$} we have
\begin{align}
\inprod{v(t)-w(t)-\int_0^t \mathrm{e}^{(t-s)\blacktriangle}F(v(s)-w(s))\mathrm{d}s}{Y^1_{\ell m}Y^2_{\ell' m'}}_{L^2(\mS^2\times\mS^2)}=0.
\end{align}
Thus, $v(t)-w(t)=\int_0^t \mathrm{e}^{(t-s)\blacktriangle}F(v(s)-w(s))\mathrm{d}s$ in $L^2(\mS^2\times\mS^2)$ for $t>0$ and from the assumption $v(0)-w(0)=0$.
Hence, $u:=v-w$ is the mild solution of the problem \eqref{eq:IVP} with the zero-initial condition. Thus, in view of Propositions \ref{prop:indep then zonal} and \ref{prop:zonal then indep} and  above, $u$ is zonal if and only if $\eta_{\lambda \nu}$ is independent of $\nu$, and so is $u$. 
\end{proof}
\begin{remark}
The above result corresponds to the case $g(x)=x$. The case $g(x)=b$ for $b\in\mathbb{R}$ corresponds to the case where $F$ as in \eqref{eq:def F} is replaced by the constant operator
\begin{align}
v(\bx_1,\bx_2)\mapsto c \sum_{\lambda=0}^\infty\sum_{\nu=-\lambda}^{\lambda}\!{A}_{\lambda}
\eta_{\lambda\nu}^2
Y_{\lambda \nu}(\bx_1)Y_{\lambda \nu}(\bx_2),\quad (\bx_1,\bx_2)\in\mS^2\times\mS^2.
\end{align}
Thus, the argument above is readily applicable. 
 For $g(x)=ax+b$ for some $a,b\in\mathbb{R}$, each term in the right hand side of \eqref{eq:E fourier coeff expansion} reads
\begin{align}
\mathbb{E}\bigg[
\int_{0}^t
\inprod{\sqrt{A}_{\lambda}(aX(s)+b)\mathrm{e}^{-\lambda_{\nu}(t-s)}Y_{\lambda \nu}}{Y_{\ell m}}
\inprod{\sqrt{A}_{\lambda} (aX(s)+b)\mathrm{e}^{-\lambda_{\nu}(t-s)}Y_{\lambda \nu}}{Y_{\ell' m'}}
\mathrm{d}s\bigg].
\end{align}
Then, the term that corresponds to the cross term $ax\cdot b$ is
\begin{align}
ab\int_{0}^t
\inprod{\sqrt{A}_{\lambda}\mathbb{E}[X(s)]\mathrm{e}^{-\lambda_{\nu}(t-s)}\eta_{\lambda\nu}Y_{\lambda \nu}}{Y_{\ell m}}
\inprod{\sqrt{A}_{\lambda}\mathrm{e}^{-\lambda_{\nu}(t-s)}\eta_{\lambda\nu}Y_{\lambda \nu}}{Y_{\ell' m'}}
\mathrm{d}s\bigg].
\end{align}
Since $\mathbb{E}[X(s)]=S(t)\xi$ is constant over $\mS^2$ given that $\xi$ is, it suffices to consider the forcing term that is a constant operator. Hence, the problem reduces to the case $g(x)=x$ and $g(x)=\text{constant}$. 
\end{remark}
\section{Discretization}\label{sec:discrete}
In this section, we will discuss a discretization of the SPDE defined as in \eqref{equ:main}.
Firstly, we consider the semi-discrete problem, in which only spatial discretization is concerned.
Then, we move on to a {fully} discrete scheme, in which the time evolution in the equation is discretized using a non-uniform implicit Euler--Maruyama scheme.
%\subsection{Semi-discrete problem}
%For any choice of 
Let $L$ and $\Lambda$ be two given non-negative integers.
An It\^{o}--Galerkin approximation %{of the semi-discrete problem}
$X^L = (X^L(t))_{t\in [0,1]}$ to $X$
%\footnote{{Doesn't ``It\^{o}--Galerkin approximation of the semi-discrete problem $X^L = (X^L(t))_{t\in [0,1]}$'' sound confusing? $X^L = (X^L(t))_{t\in [0,1]}$ \textit{is} an It\^{o}--Galerkin approximation, which is also semi-discrete.}}
is defined by
\begin{equation}\label{equ:defXL}
 X^L(t) = 
 \sum_{\ell=0}^L \sum_{|m|\le \ell}  
 X^L_{\ell m}(t) Y_{\ell m}
\end{equation}
with {real}-valued processes $X^L_{\ell m} = (X^L_{\ell m}(t))_{t\in [0,1]}$ that 
solve the finite-dimensional system
\begin{equation}\label{equ:semi}
\begin{aligned}
 \mathrm{d}{X}^L_{\ell m}(t) &=   
 -\mu_\ell X^L_{\ell m}(t) \mathrm{d}t
 + \sum_{\ell'=0}^{\Lambda}
 \sum_{ m'=-\ell'}^{\ell'}
 \sqrt{A_{\ell'}}\inprod{B(X^L(t)) Y_{\ell'm'}}{ Y_{\ell m}} 
   \mathrm{d} w_{\ell' m'}(t) \\
 X^L_{\ell m}(0) &= \inprod{\xi}{Y_{\ell m}} 
\end{aligned}
\end{equation}
%%%%------------------------------------------------------------------
%\subsection{Fully discrete problem}

For a fully discrete problem, 
let us first discretize the interval $[0,1]$ with a uniform partition,
i.e., we partition the interval with $t_k = k/n$, for 
$k=0,1,2,\ldots n$.
An implicit Euler--Maruyama scheme with uniform step-size $1/n$ 
being applied to \eqref{equ:semi} is given by
\begin{align*}
\what{X}^L_{\ell m}(t_{k}) &= \what{X}^L_{\ell m}(t_{k-1})
- \mu_{\ell} \what{X}^L_{\ell m}(t_{k}) \frac{1}{n} +\\
&\sum_{\ell'=0}^{\Lambda} 
 \sum_{m'=-\ell'}^{\ell'}
 \sqrt{A_{\ell'}}\inprod{B(\what{X}^L(t_{k-1})) Y_{\ell'm'}}{ Y_{\ell m}}
   ( w_{\ell' m'}(t_k) - w_{\ell' m'}(t_{k-1}))
\end{align*}   
with the initial condition
\begin{equation}
 \what{X}^L_{\ell m}(0) = \inprod{\xi}{Y_{\ell m}}.
\end{equation}
More generally, we can use a non-uniform scheme\YK{: 
it is known that non-uniform time discretizations 
can lead to asymptotically optimal approximations that cannot be achieved by uniform ones in general. 
See \cite[Section 5]{GroRit07a}, also \cite[Remark 6]{GroRit07}. 
}
As proposed by \cite{GroRit07,GroRit07a}, we evaluate the Brownian motion $w_{\ell m}$
with step-size $1/n_{\ell'}$ depending on $\ell' = 0,\ldots,\Lambda$. Let
\begin{align}\label{equ:timegrid}
  t_{k,\ell} = k/n_\ell, \quad k=0,\ldots,n_\ell.
\end{align}
We define
\[
  0 = \tau_0 < \dots < \tau_K = 1
\]
by
\[
 \{\tau_0,\ldots,\tau_K \} =
  \bigcup_{ \ell'=0}^{\Lambda} \{ t_{0,\ell'},\ldots,t_{n_\ell',\ell'} \}.
\]
Let
\[
\calK_k =\{ \ell' \in \{0,1,\ldots,\Lambda\} : 
\tau_k \in \{t_{0,\ell'},\ldots,t_{n_{\ell'},\ell'}  \} \},
\]
for $k=0,\ldots,K$ and we define $s_{k,\ell'}$ for $k=1,\ldots,K$ and $\ell'=0,\ldots,\Lambda$ 
by
\[
  s_{k,\ell'} = \max (\{t_{0,\ell'},\ldots,t_{n_{\ell'},\ell'} \}\cap [0,\tau_k) ).
\]
We use the following approximation of the \YK{eigenvalues of the} semigroup generated by $\Delta^{\ast}$
\begin{equation}\label{def:Gamma}
\Gamma_{\ell}(t) = \prod_{\nu=1}^K \frac{1}
{1 + \mu_\ell(t\wedge \tau_\nu - t\wedge \tau_{\nu-1})}.
\end{equation}
The drift-implicit Euler scheme is given by, if $t \in (\tau_{k-1},\tau_k]$,
\begin{equation}\label{equ:drift Euler}
\begin{aligned}
\what{X}^L_{\ell m}(t) &= 
  \frac{\Gamma_{\ell}(t)}{\Gamma_\ell(\tau_{k-1})}
 \left(
 \widehat{X}^L_{\ell m}(\tau_{k-1}) \right.\\
&\left. \quad + \sum_{\ell' \in \calK_k}
  \sum_{|m'|\le \ell'} \sqrt{A_{\ell'}}\inprod{B(\what{X}^L(s_{k,\ell'}) ) Y_{\ell'm'}}{Y_{\ell m}}
    \frac{\Gamma_\ell(\tau_{k-1})}{\Gamma_\ell(s_{k,\ell'})}
    ( w_{\ell'm'}(\tau_k) - w_{\ell' m'}(s_{k,\ell'}))
 \right)
\end{aligned}
\end{equation}
Equivalently, for $t \in (\tau_{k-1},\tau_k]$, we have
\begin{equation}\label{equ:recur}
\begin{aligned}
\what{X}^L_{\ell m}(t) &= \Gamma_\ell(t) \inprod{\xi}{Y_{\ell m}}\\
&+ \sum_{\ell'=0}^{\Lambda} \sum_{m'=-\ell'}^{\ell'}
\sum_{t_{j,\ell'} \le \tau_k}
 \sqrt{A_{\ell'}}\inprod{B(\what{X}^L(t_{j-1,\ell'})) Y_{\ell'm'}}{ Y_{\ell m}}
 \frac{\Gamma_\ell(t)}{\Gamma_\ell(t_{j-1,\ell'})} \\
& \times     ( w_{\ell' m'}(t_{j,\ell'}) - w_{\ell' m'}(t_{j-1,\ell'})).
 \end{aligned}
\end{equation}

Hence, a fully discrete solution to \eqref{equ:semi} {with a non-uniform
time discretization} is defined by
\begin{equation}\label{eqn:full discrete}
\what{X}^L(t) = \sum_{\ell=0}^{L} \sum_{|m|\le\ell}
  \what{X}^L_{\ell m}(t) Y_{\ell m},
\end{equation}
where the coefficients $\what{X}^L_{\ell m}(t)$ are given as in \eqref{equ:drift Euler}.

%%%%%%%%%%%%%%%%%%%%%%%%%%%%%%%%%%%%%%%%%%%%%%%%%%%%%%%%%%%%%%%%
\section{Error analysis}\label{sec:err anal}
%Let us define the Sobolev space $H^1$ on the sphere $\mS^2$
%\[
%  H^1 := \left\{ h \in L^2(\mS^2) : \|h\|^2_{H^1} =\sumlm (1+\mu_\ell) 
%  \inprod{h}{Y_{\ell m}}^2 < \infty \right\}.
%\]
%and its subspace
%\[
% H^1_0 := \left\{ h \in L^2(S) : \int_{\mS^2} h dS = 0,\;\;
% \|h\|^2_{H^1_0} = \sumlm \mu_\ell
%  \inprod{h}{Y_{\ell m}}^2 < \infty \right\}.
%\]
%Since $\mu_\ell  \ge 2$ for all $\ell>0$ we 
%have the following inequality
%\begin{equation}\label{equ:poincare}
% \|h\|^2_{H^1} \le 2 \|h\|^2_{H^1_0} \text{ for all } h \in H^1_0
%\end{equation}
\begin{comment}
\footnote{This lemma is not used.}
\begin{lemma}\label{lem:tail}
We have 
\[
\int_0^1 \bbE \|X(t)\|^2_{H^1} \mathrm{d}t < \infty.
\]
In particular, for Lebesgue-almost every $t\in [0,1]$, we have $X(t) \in H^1$ with
probability $1$.
\end{lemma}
\begin{proof}
The claim follows from \eqref{equ:EXlm},
\eqref{equ:Gamma2} and \eqref{equ:psiL1}.
\end{proof}

\end{comment}
We need the following lemma for the error estimate.
\begin{lemma}\label{lem:Ylpmplm id}
Let $f\in H$. Then, for any $\ell'\in \{0,\dotsc,\Lambda\}$ we have
\begin{align}
\sum_{|m'|\le\ell'}  \sumlm
 \langle fY_{\ell',m'}, Y_{\ell,m}\rangle^2
=
\frac{2\ell'+1}{4\pi}{\|f\|^2}.
\end{align}
\end{lemma}
\begin{proof}
 For each $\ell'\in \{0,\dotsc,\Lambda\}$, $|m'| \le \ell'$, we have
\begin{align}
\sumlm 
 \langle fY_{\ell',m'}, Y_{\ell,m}\rangle^2
=&
\sumlm 
  (\widehat{fY_{\ell',m'}})_{\ell m}^2
=
\| fY_{\ell',m'} \|^2 \\
=&
\int_{\mS^2} |f(\bx)|^2Y_{\ell',m'}(\bx){Y_{\ell',m'}(\bx)} \mathrm{d}\varsigma(\bx).
\end{align}
From the addition theorem, it follows that
\begin{align}
\sumlm &
	\frac{4\pi}{2\ell'+1}
		\sum_{|m'|\le \ell' }
		 \langle fY_{\ell',m'}, Y_{\ell,m}\rangle^2 \\
&=
\int_{\mS^2} |f(\bx)|^2 
	\frac{4\pi}{2\ell'+1}
		\left(
			\sum_{|m'|\le \ell' }
			Y_{\ell',m'}(\bx){Y_{\ell',m'}(\bx)}
		\right)
			\mathrm{d}\varsigma(\bx) \\
&=
\int_{\mS^2} |f(\bx)|^2 \mathrm{d}\varsigma(\bx).
\end{align}
\end{proof}
Now, we obtain the following spatial truncation error. 
%Note that since the Galerkin equations satisfy the same condition as the continuous equation, similarly to the conclusion reached in \eqref{equ:Ebound}, by using the same arguments as in Section 3 we also have
%\footnote{
%{``by using the same argument as in Section 3'' sounds a bit off. Maybe it should be deleted? To say (\ref{equ:Ebound}), no argument is made, we just quoted Da Prato's book.}
%}
{From the result} \cite[Section 7.1]{DaPZab14} {together with the discussion to derive} \cite[(6.8)]{GroRit07a}, 
{similarly to} \eqref{equ:Ebound} {we have}
\begin{equation}\label{XLbounded}
\sup_{t\in [0,1]} \bbE \| X^L(t)\|^2 \le C_1.
\end{equation}
\begin{theorem}\label{thm:semidiscrete}
\YK{Let $B$ be defined by \eqref{equ:def B}. Then, for} $L,\Lambda >0$, with the definition
$X^L$ as in \eqref{equ:semi} we have the following estimate:
\begin{align}
\bbE \bigg( \int_0^1 \|X(t) - X^L(t)\|^2 \mathrm{d}t\bigg)
\le C \bigg( \frac{1}{L^2} + \sum_{\ell' > \Lambda}
\sum_{|m'|\le \ell'}
A_{\ell'}\bigg)
\end{align}
%\begin{empheq}[box=\shadebox]{align*}
%\bbE \left( \int_0^1 \|X(t) - X^L(t)\|^2 \mathrm{d}t\right)
%\le C \left( \frac{1}{L^2} + \sum_{\ell' > \Lambda}
%{\sum_{|m'|\le \ell'}}
%\frac{A_{\ell'}}{\ell'}\right)
%\end{empheq}
\end{theorem}
\begin{proof}
%\footnote{(See M\"{u}ller-Gronbach--Ritter (2007), BIT, pp. 408--)}
Using \eqref{equ:defX},\eqref{equ:defXlm} and \eqref{equ:defZlm} we can write
\begin{align*}
X(t) &= \sum_{\ell=0}^L \sum_{m=-\ell}^\ell X_{\ell m}(t) Y_{\ell m}
   + \underbrace{\sum_{\ell > L} \sum_{m=-\ell}^\ell X_{\ell m}(t) Y_{\ell m}}_{R_L(t)} \\
   &= A_L^{(1)}(t) + A_L^{(2)}(t) + R_L(t),
\end{align*}
with
\begin{align*}
 A_L^{(1)}(t) &:=  
 \sum_{\ell=0}^L \sum_{m=-\ell}^\ell
 \left( \exp(-\mu_\ell t) \inprod{\xi}{Y_{\ell m}}
 + \sum_{\ell'=0}^{\Lambda} 
 \sum_{m'=-\ell'}^{\ell'}
 \sqrt{A_{\ell'}}
 Z_{\ell' m',\ell m}(t) \right) Y_{\ell m},\\
A_L^{(2)}(t) &:= 
  \sum_{\ell=0}^L \sum_{m=-\ell}^\ell
 \sum_{\ell'>\Lambda} \sum_{m'=-\ell'}^{\ell'}
 \sqrt{A_{\ell'}}
 Z_{\ell' m',\ell m}(t) Y_{\ell m},
\end{align*}
{where $Z_{\ell' m',\ell m}$ is defined as in} \eqref{equ:defZlm}.
With the solution $X^L$ of the semi-discrete problem \eqref{equ:defXL},
we have
\begin{align*}
\int_0^{t} \bbE \|X(s) - X^L(s)\|^2 \mathrm{d}s &\preceq 
\int_0^{t} \bbE \|A^{(1)}_L(s) - X^L(s)\|^2 \mathrm{d}s
+ \int_0^1 \bbE \|A^{(2)}_L(t)\|^2 \mathrm{d}t + \\
&\qquad\int_0^1 \bbE \|R_L(t)\|^2 \mathrm{d}t.
\end{align*}
We have
\[
\bbE(Z_{\ell'm',\ell m}(t)) =
\bbE\int_{0}^t \exp(-\mu_\ell(t-s))
  \inprod{B(X(s)) Y_{{\ell' m'}}}{ Y_{{\ell m}}} \mathrm{d} w_{\ell' m'}(s),
\]
It\^{o} isometry yields
\[
\bbE(Z_{\ell'm',\ell m}(t))^2 =
\int_0^t \exp(-2\mu_\ell(t-s)) 
 \bbE |\inprod{B(X(s)) Y_{{\ell' m'}}}{ Y_{{\ell m}}}|^2 \mathrm{d}s,
\]
and thus
\begin{align*}
\int_0^1 
\bbE (Z_{\ell' m',\ell m}(t))^2 \mathrm{d}t
&\preceq 
\int_0^1 \bbE ( \inprod{B(X(t)) Y_{\ell' m'}}{Y_{\ell m} })^2 \mathrm{d}t\\
&={\eta_{\ell m}^2\int_0^1 \bbE ( \inprod{T_g(X(t)) Y_{\ell' m'}}{Y_{\ell m} })^2 \mathrm{d}t}.
\end{align*}
Therefore, in view of Lemma \ref{lem:Ylpmplm id} 
{for all $\ell'\ge 0$ we have}
\begin{align}
&{
\sumlm A_{\ell'}\sum_{|m'|\le \ell'}
\int_0^1 
\bbE (Z_{\ell' m',\ell m}(t))^2 \mathrm{d}t}
\\&
{\preceq 
\sup_{\mu\nu}\eta_{\lambda\nu}^2
\int_0^1 
A_{\ell'}\sum_{|m'|\le \ell'}\sumlm 
\bbE ( \inprod{T_g(X(t)) Y_{\ell' m'}}{Y_{\ell m} })^2 \mathrm{d}t}\\
&
{\preceq 
\sup_{\mu\nu}\eta_{\lambda\nu}^2
\int_0^1 
A_{\ell'}
\frac{2\ell'+1}{4\pi}\bbE\|T_g(X(t))\|^2 \mathrm{d}t.}
\end{align}
{Hence, from} {\eqref{equ:Tgu H bd}} {and} {\eqref{equ:Ebound}} {we obtain}
%\footnote{{I think the last line $\preceq\sum_{\ell' > \Lambda} \sum_{|m'|\le \ell'} {A_{\ell'}}$ can be deleted now.}}
\begin{align*}
\int_0^1 \bbE\|A^{(2)}_L(t)\|^2 \mathrm{d}t
&\le \sumlm 
\sum_{\ell'>\Lambda}
	A_{\ell'}
\sum_{|m'|\le{\ell'}}
\int_0^1 
\bbE(Z_{\ell m,\ell' m'}(t))^2 \mathrm{d}t \\
& \le
\sum_{\ell'>\Lambda}{\frac{2\ell'+1}{4\pi}}
	A_{\ell'}
\int_0^1 
(1 + \bbE\|X(t)\|^2 )
\mathrm{d}t % \\
%& \preceq\sum_{\ell' > \Lambda} \sum_{|m'|\le \ell'} {A_{\ell'}}
\le c.
\end{align*}
From \eqref{equ:muXlm}, we have
\[
\int_0^1 \bbE \|R_L(t)\|^2 \mathrm{d}t=
 \sum_{\ell > L} {\sum_{|m|\le \ell}} \int_0^1 \bbE (X_{\ell m}(t))^2 \mathrm{d}t
 \preceq \frac{1}{L^2} \le c.
\]
{We next see that for $\ell\in\{0,\dots,L\}$, $m\in\{-\ell,\dotsc,\ell\}$ we have}
\begin{align}
{
\bbE(X_{\ell m}(t) - X^{L}_{\ell m}(t))^2
=
\sum_{\ell'=0}^{\Lambda} \sum_{m'=-\ell'}^{\ell'} 
A_{\ell'} \eta_{\ell'm'}^2\int_{0}^{t}\bbE\inprod{\big(T_g(X(s)) - T_g(X^L(s))\big)Y_{\ell'm'}}{Y_{\ell m}}^2\mathrm{d}s.
}
\end{align}
Thus, from $\|T_g(u)-T_g(v)\|\le\|g'\|_{\infty}\|u-v\|$ ($u,v\in H$) {we have}
\begin{align*}
\bbE \| A^{(1)}_L(t) - X^L(t)\|^2
%&\preceq \int_0^t \bbE \|B(X(s)) - B(X^L(s))\|^2 \mathrm{d}s} \\
% should be either $\|B(X(s)) - B(X^L(s))\|_{HL}$ or the following
&\preceq \int_0^t \bbE \|T_g(X(s)) - T_g(X^L(s))\|^2 \mathrm{d}s \\
&\preceq \int_0^t \bbE \|X(s) - X^{L}(s)\|^2 \mathrm{d}s\\
&\preceq 2c + \int_0^t \bbE \|A^{(1)}_L(s) - X^L(s)\|^2 \mathrm{d}s,
\end{align*}
{where  $X(t)-X^{L}(t) =  A_L^{(2)}(t) + R_L(t) + A_L^{(1)}(t) -X^{L}(t)$ is used in the last line.}
Since $\bbE( \langle{A^{(1)}_L(t),A^{(2)}_L(t)}\rangle ) = 0$,
we get $\bbE\| A^{(1)}_L(t)\|^2 \le \bbE \|X(t)\|^2$.
Using \eqref{equ:Ebound} and \eqref{XLbounded} we conclude that
\[
 \sup_{ t\in [0,1] } \bbE \| A^{(1)}_L(t) - X^L(t)\|^2 < \infty.
\]
The proof is completed by applying Gronwall's Lemma.
\begin{comment}
\footnote{
We have
\begin{align*}
\phi(t)&:=\int_0^t \bbE \|X(s) - X^{L}(s)\|^2 \mathrm{d}s
\preceq \hbox{\sout{$t$}}\sum A_{\ell'}\ell'+ \hbox{\sout{$t$}}\frac1{L^2}+ \int_0^t\int_0^r \bbE \|X(s) - X^{L}(s)\|^2 \mathrm{d}s \mathrm{d}r
\end{align*}
and 
\begin{align*}
\int_0^t\int_0^r \bbE \|X(s) - X^{L}(s)\|^2 \mathrm{d}s \mathrm{d}r
&\preceq 2c + \int_0^r\sup_{t\in [0,1]} \bbE \|A^{(1)}_L(s) - X^L(s)\|^2 \mathrm{d}s<\infty.
\end{align*}
Thus,
$\phi(t)\le(\sum A_{\ell'}\ell'+ \frac1{L^2})e^{t}$ $(t\in[0,1])$
}
\end{comment}
\end{proof}
%-----------------------------------------------

%-------------------------------------------------------------
In the following lemma, we discretize the time interval 
$[0,1]$ using a uniform partition of length $1/k$ and 
provide an error estimate.

\begin{lemma}\label{lem:piecewise}
%\footnote{See BIT p. 407, Lemma 6.2}
For ${k}\in \N$, with $X^L$ being defined as in \eqref{equ:defXL}, %there holds 
%\footnote{{I have read ``there holds'' is not really English.}}
{we have the following upper bound}
\[
\sum_{j=0}^{k-1} \int_{j/k}^{(j+1)/k} 
\bbE \| X^L(t) - X^L(j/k) \|^2 \mathrm{d}t \preceq 1/k.
\]
\end{lemma}
%---------------
\begin{proof}
The results of Lemma~\ref{lem:Xcts} are valid for $X^L$,
 with $\psi$ being replaced by
\[
\overline{\psi}(t) = \sum_{\ell \le L} 
 \sum_{|m| \le \ell}\mu_\ell \bbE( (X^L_{\ell,m}(t))^2 ).
\]
For $j \in \{0,\ldots,k-1\}$ take
\begin{comment}
\footnote{
Suppose $ \int_{j/k}^{(j+1)/k} \overline{\psi}(t) \mathrm{d}t<\overline{\psi}(s)/{k}$ for any $s\in [j/k, (j+1)/k]$. Then, integrating both sides yields
$$
\frac1k\int_{j/k}^{(j+1)/k} \overline{\psi}(t) \mathrm{d}t
<
\frac1k\int_{j/k}^{(j+1)/k} \overline{\psi}(s) \mathrm{d}s,
$$
contradiction.
}
\end{comment}
 $s_j \in [j/k, (j+1)/k]$ with
\[
\frac{\overline{\psi}(s_j)}{k} \le \int_{j/k}^{(j+1)/k} \overline{\psi}(t) \mathrm{d}t.
\]
On the first subinterval, we have
\[
  \int_0^{1/k} \bbE \|X^L(t) - X^L(0)\|^2 \mathrm{d}t
  \le \frac{2}{k} \sup_{ t \in [0,1] } \bbE \| X^L(t)\|^2 \preceq 1/k.
\]
On the subintervals $[j/k,(j+1)/k]$ with $j \ge 1$ we estimate as follows.
If $t\in [j/k,s_j]$, then
\begin{align*}
\bbE\| X^L(t) - X^L(j/k) \|^2 
 &\preceq \bbE \| X^L(t) - X^L(s_{j-1})\|^2
  + \bbE \| X^L(s_j) - X^L(j/k)\|^2 \\
 & \preceq \frac{1}{k} (1 + \overline{\psi}(s_{j-1})) \\
 & \preceq \frac{1}{k} + \int_{(j-1)/k}^{j/k} \overline{\psi}(s) \mathrm{d}s.
\end{align*}
If $t \in [s_j, (j+1)/k]$, then
\begin{align*}
&\bbE\|X^L(t) - X^L(j/k)\|^2 \\
&\le \bbE\|X^L(t)-X^L(s_j)\|^2 + 
     \bbE\|X^L(s_j) - X^L(s_{j-1})\|^2 +
     \bbE\|X^L(s_{j-1}) - X^L(j/k)\|^2 \\
&\preceq \frac{1}{k}( 1+ \overline{\psi}(s_j) + \overline{\psi}(s_{j-1}) ) \\
&\preceq \frac{1}{k} + \int_{(j-1)/k}^{(j+1)/k} \overline{\psi}(s) \mathrm{d}s.
\end{align*}
Hence, we conclude that
\[
\int_{j/k}^{(j+1)/k}
\bbE \|X^L(t) - X^L(j/k)\|^2 \mathrm{d}t
\preceq \frac{1}{k^2} + 
\frac{1}{k} \int_{(j-1)/k}^{(j+1)/k} \overline{\psi}(s)ds,
\]
from which the result follows.
\end{proof}
%%%%%%%%%%%%%%%%%%%%%%%%%%%%%%%%%%%%%%%%%%
We record the following estimates for the properties regarding the spectral representations of resolvents by M\"{u}ller-Gronbach and Ritter \cite{GroRit07a}.
\begin{lemma}\label{lem:semigroup}
Suppose $\ell \le L$ and $\ell' \le \Lambda$. Then, for $j=0,\ldots,n_{\ell'}-1$,
\[
\int_{ t_{j,\ell'} }^1 \frac{\Gamma^2_\ell(t)}{ \Gamma^2_\ell( t_{j,\ell'} )} \mathrm{d}t \le 2/\mu_\ell
\]
as well  as
\[
\int_{t_{j,\ell'}}^1 \left( \frac{\Gamma_\ell(t)}{\Gamma_\ell( t_{j,\ell'} )}
-\exp(-\mu_\ell(t- t_{j,\ell'} ))\right)^2 \mathrm{d}t \preceq 1/n^*,
\]
where $n^* = \max\{ n_\ell: \ell=0,\ldots,\Lambda \}$.
Furthermore, for $0\le s \le t \le 1$,
\[
\left| 1 - \frac{\Gamma_\ell(t)}{\Gamma_\ell(s)} \right|
\le \min(1, \mu_\ell(t-s)).
\]
\end{lemma}
\begin{proof}
The statement follows from \cite[Lemma 6.3]{GroRit07a}.
\end{proof}
%%%%%%%%%%%%%%%%%%%%%%%%%%%%%%%%%%%%%%%
\begin{comment}
{
We assume $\what{X}^L$ uses a total of $N$ evaluations of scalar Brownian motions $(w_{\ell'm'})$.} blue{Clearly, we have}
$$
blue{\Lambda+1\le N}.
$$
%
%\sout{with a function $\mathfrak{N}$.}
{ Assume further that $A_\ell\sim \ell^{-\alpha}$ ($\alpha>0$). (Maybe assume this somewhere before). If} blue{$\alpha\in(0,2]$} {(not sure this can happen as we assume trace class but I will keep this for now)}
%\sout{, with a function $C_1(\ell',N)$ such that $\sup\limits_{\ell'\in\mathbb{N}_0,N\in\mathbb{N}}1/C(\ell',N)\preceq 1$ we take}
%\sout{$n_{\ell'}=\lceil \sqrt{A_{\ell'}} \Lambda^{1-\frac{\alpha}{2}}  C_1(\ell',N) \rceil)$.}
$$
blue{n_{\ell'}=\lceil \sqrt{A_{\ell'}} N^{\delta}\rceil),}
$$
blue{where $\delta=\delta(\alpha)\ge 2-\frac{\alpha}{2}$.}
%
If $\alpha\in(blue{2},\infty)$, we take 
%
%\sout{$n_{\ell'}=\lceil \sqrt{A_{\ell'}} \Lambda C_2(\ell',N) \rceil),$}
%\sout{where the function $C_2$ satisfies $\sup\limits_{\ell'\in\mathbb{N}_0,N\in\mathbb{N}}1/C_2(\ell',N)\preceq 1$}.
%
%
$$
blue{n_{\ell'}=\lceil \sqrt{A_{\ell'}} N^\delta \rceil),}
$$
blue{with $\delta=\delta(\alpha)\ge 1$.}
\end{comment}
The following lemma is important to justify the use of the non-uniform step size in Theorem \ref{thm:main result}. 
\begin{lemma}\label{lem:bd D}
{Let the operator $B$ be defined by} \eqref{equ:def B}.
Then, for any $u\in H$ we have
\begin{align}
\sumlm 
	\sum_{\ell'=0}^\Lambda \sum_{ |m'|\le\ell' }
		\langle {B}(u) Y_{\ell',m'}, Y_{\ell,m}\rangle^2
\frac{A_{\ell'}}{n_{\ell'}}
\preceq
	1 + \|u\|^2.
\end{align}
\end{lemma}
\begin{proof}
{From} Lemma \ref{lem:Ylpmplm id}, for each $\ell'\in \{0,\dotsc,\Lambda\}$,
%{let $f=B(u)$},
%\footnote{{``$B(u)$'' is an operator, not a function.}}
 we have
\begin{align}
\sumlm \sum_{|m'|\le\ell'} 
 \langle fY_{\ell',m'}, Y_{\ell,m}\rangle^2
%=&
%\int_{\mS^2} |f(\bx)|^2
%\sum_{|m'|\le\ell'}Y_{\ell',m'}(\bx){Y_{\ell',m'}(\bx)} \mathrm{d}\varsigma(\bx)\\
=&
\frac{2\ell'+1}{4\pi}\int_{\mS^2} |f(\bx)|^2 \mathrm{d}\varsigma(\bx),
\quad\text{{for any $f\in H$}.}
\end{align}
Thus, multiplying $\frac{A_{\ell'}}{n_{\ell'}}$ to the both sides and summing over $\ell'$ yields.
\begin{align}
\sumlm 
	\sum_{\ell'=0}^\Lambda \sum_{ |m'|\le\ell' }
		\langle f Y_{\ell',m'}, Y_{\ell,m}\rangle^2
\frac{A_{\ell'}}{n_{\ell'}}
&=
\sum_{\ell'=0}^\Lambda
	\frac{2\ell'+1}{4\pi}
		\| f \|^2 \frac{A_{\ell'}}{n_{\ell'}}.\label{equ:bd Dlmlm'}
\end{align}
Since 
${\langle {B}(u) Y_{\ell',m'}, Y_{\ell,m}\rangle^2
=
\eta_{\ell' m'}^2\langle T_{g}(u) Y_{\ell',m'}, Y_{\ell,m}\rangle^2}$, 
{in view of} {\eqref{equ:Aell summable}}, \eqref{equ:Tgu H bd} 
and $\sup_{\ell,m}|\eta_{\ell m}|<\infty$ the statement follows.
\end{proof}
{The following lemma is needed in the error analysis of the fully discrete solution}.
\begin{lemma}\label{lem:bounded}
{For the fully discrete solution $\what{X}^L$ defined as in} \eqref{eqn:full discrete}, 
{we have the following upper bound}
\[
\sup_{t \in [0,1]} \bbE \|\what{X}^L(t)\|^2 \preceq 1.
\]
\end{lemma}
\begin{proof}
Following \cite{GroRit07a}, we introduce the process which continuously interpolates the {noise} of $\what{X}^L(t)$,
\[
\wtd{X}^{L}(t) = \sum_{\ell=0}^L \wtd{X}_{\ell m}(t) Y_{\ell m} 
\]
with $\wtd{X}_{\ell m}(0) = \inprod{\xi}{Y_{\ell m}}$ and 
\[
\begin{aligned}
\wtd{X}_{\ell m}(t) 
 &= \frac{\Gamma_j(t)}{\Gamma_j(\tau_{k-1})}
\left(
\wtd{X}_{\ell m}(\tau_{k-1}) \right.\\
&\left. \quad + \sum_{\ell' \in \calK_k}
  \sum_{|m'|\le \ell'} \sqrt{A_{\ell'}} \inprod{B(\wtd{X}^L(s_{k,\ell'}) ) Y_{\ell'm'}}{Y_{\ell m}}
    \frac{\Gamma_\ell(\tau_{k-1})}{\Gamma_\ell(s_{k,\ell'})}
    ( w_{\ell'm'}(t) - w_{\ell' m'}(s_{k,\ell'}))
 \right)
\end{aligned}
\]
for $t \in (\tau_{k-1},\tau_k]$.
{In comparison with the} equation \eqref{equ:drift Euler}, $\wtd{X}^L$ is obtained
from $\what{X}^L$ by replacing the Brownian increments 
$w_{\ell' m'}(\tau_k) - w_{\ell' m'}(s_{k,\ell'})$ 
by $w_{\ell' m'}(t) - w_{\ell' m'}(s_{k,\ell'})$.

Note that $\what{X}^L_{\ell m}$ and $\wtd{X}_{\ell m}$ as well as
$\what{X}^L$ and $\wtd{X}^L$ coincide at the points $\tau_k$. Moreover,
by the construction of these processes we have
$\wtd{X}_{\ell m}(\tau_k)$ and $\wtd{X}^L(\tau_k)$ are measurable
with respect to
$${\mathcal{G}_{k}:=
\sigma(\{w_{\ell' m'}(t_{j,\ell'}): t_{j,\ell'} \le \tau_k, \ell'\le \Lambda, |m'|\le \ell'\})}.$$ %\sout{$\sigma(\{a_{\ell m}(t_{k,\ell'}): t_{k,\ell'} \le \tau_k, \ell'\le \Lambda\})$}.

Thus, if $t \in (\tau_{k-1},\tau_k]$, we obtain
\begin{align*}
&\bbE( \wtd{X}_{\ell m}(t) - \wtd{X}_{\ell m}(\tau_{k-1}) )^2 \\
&{
=\bbE\bigg(\left( 1 - \frac{\Gamma_\ell(t)}{\Gamma_\ell(\tau_{k-1})} \right)
    \wtd{X}_{\ell m}(\tau_{k-1}) }\\
&+\frac{\Gamma_\ell(t)}{\Gamma_\ell(\tau_{k-1})}
		\sum_{\ell' \in \calK_k } \sum_{|m'|\le \ell'}
		\sqrt{A_{\ell'}} \inprod{B(\wtd{X}^L(s_{k,\ell'}) ) Y_{\ell'm'}}{Y_{\ell m}}
    \frac{\Gamma_\ell(\tau_{k-1})}{\Gamma_\ell(s_{k,\ell'})}
    ( w_{\ell'm'}(t) - w_{\ell' m'}(s_{k,\ell'}))
  \bigg)^2.
\end{align*}
Now, from the definition of $s_{k,\ell'}$ for $\ell'\in\calK_k$, and $\{\tau_0,\dotsc,\tau_K\}$,
we have $\tau_{k_0}=s_{k,\ell'}$
for some $k_0\in\{0,\dotsc,k-1\}$.
 Thus, for each $(\ell',m')$ % from the properties of the standard Brownian motions.
\begin{align}
&
\bbE\bigg[ \bbE\bigg[
    \wtd{X}_{\ell m}(\tau_{k-1})
		\inprod{B(\wtd{X}^L(s_{k,\ell'}) ) Y_{\ell'm'}}{Y_{\ell m}}
    {( w_{\ell'm'}(t) - w_{\ell' m'}(s_{k,\ell'}))}
	\bigg|\,\mathcal{G}_{k_0}\bigg]\bigg]\nonumber
\\
&=
\bbE\bigg[
		    {( w_{\ell'm'}(t) - w_{\ell' m'}(s_{k,\ell'}))}
		\bbE\bigg[
		    \wtd{X}_{\ell m}(\tau_{k-1})
		    \inprod{B(\wtd{X}^L(s_{k,\ell'}) ) Y_{\ell'm'}}{Y_{\ell m}}
			\bigg|\,\mathcal{G}_{k_0}\bigg]
\bigg]
\\
&=
\bbE\bigg[
	{w_{\ell'm'}(t) - w_{\ell' m'}(s_{k,\ell'})}
\bigg]
\bbE\bigg[\bbE\bigg[
		    \wtd{X}_{\ell m}(\tau_{k-1})
		    \inprod{B(\wtd{X}^L(s_{k,\ell'}) ) Y_{\ell'm'}}{Y_{\ell m}}
			\bigg|\,\mathcal{G}_{k_0}
\bigg]\bigg]=0.
\end{align}
Further, from $\mathcal{G}_{k_0}$-measurability of $\wtd{X}^L(s_{k,\ell'})$ % and the property of Brownian motion,
 we have
\begin{align*}
&
\bbE\left[\inprod{B(\wtd{X}^L(s_{k,\ell'}) ) Y_{\ell'm'}}{Y_{\ell m}}^2
    ( w_{\ell'm'}(t) - w_{\ell' m'}(s_{k,\ell'}))^2\right]
\\
&=
\bbE\left[\inprod{B(\wtd{X}^L(s_{k,\ell'}) ) Y_{\ell'm'}}{Y_{\ell m}}^2
\right]
\bbE\left[
    ( w_{\ell'm'}(t) - w_{\ell' m'}(s_{k,\ell'}))^2
\right] \\
&=
\bbE\left[\inprod{B(\wtd{X}^L(s_{k,\ell'}) ) Y_{\ell'm'}}{Y_{\ell m}}^2
\right]
\bbE\left[
    ( w_{\ell'm'}(t) - w_{\ell' m'}(s_{k,\ell'}))^2
\right] \\
&=
\bbE\left[\inprod{B(\wtd{X}^L(s_{k,\ell'}) ) Y_{\ell'm'}}{Y_{\ell m}}^2
\right]
(t-s_{k,\ell'}).
\end{align*}
Thus, it follows that
\begin{align*}
&\bbE( \wtd{X}_{\ell m}(t) - \wtd{X}_{\ell m}(\tau_{k-1}) )^2\\
&=
\left( 1 - \frac{\Gamma_\ell(t)}{\Gamma_\ell(\tau_{k-1})} \right)^2
    \bbE(\wtd{X}_{\ell m}(\tau_{k-1}))^2 \\
	&+\frac{\Gamma_\ell(t)^2}{\Gamma_\ell(\tau_{k-1})^2}
		\sum_{\ell' \in \calK_k } \sum_{|m'|\le \ell'}
		\bbE\bigg[
		\inprod{B(\wtd{X}^L(s_{k,\ell'}) ) Y_{\ell'm'}}{Y_{\ell m}}^2
		\bigg]
    \frac{\Gamma_\ell(\tau_{k-1})^2}{\Gamma_\ell(s_{k,\ell'})^2}
    A_{\ell'}(t-s_{k,\ell'}) \\
&\le
\bbE(\wtd{X}_{\ell m}(\tau_{k-1}))^2
+
\sum_{\ell' \in \calK_k }
	\sum_{|m'|\le \ell'}
	D_{\ell',m',\ell,m}(s_{k,\ell'})
	\frac{\Gamma_\ell(t)^2}{\Gamma_\ell(s_{k,\ell'})^2}
	    A_{\ell'}(t-s_{k,\ell'})\\
&\le
\bbE(\wtd{X}_{\ell m}(\tau_{k-1}))^2
+
\sum_{\ell' \in \calK_k }
	\sum_{|m'|\le \ell'}
	D_{\ell',m',\ell,m}(s_{k,\ell'})
	\frac{\Gamma_\ell(s_{k,\ell'})^2}{\Gamma_\ell(s_{k,\ell'})^2}
	    A_{\ell'}(\tau_k-s_{k,\ell'})\\
&\le
 \bbE(\wtd{X}_{\ell m}(\tau_{k-1}))^2
 +
 \sum_{\ell' \in \calK_k }
	 \sum_{|m'|\le \ell'}
   	D_{\ell',m',\ell,m}( s_{k,\ell'} )
 	   \frac{ A_{\ell'}}{n_{\ell'}},
\end{align*}
\begin{comment}
\begin{empheq}[box=\shadebox]{align}
 &= \left( 1 - \frac{\Gamma_\ell(t)}{\Gamma_\ell(\tau_{k-1})} \right)^2 
   \bbE( (\wtd{X}_{\ell m}(\tau_{k-1}))^2) \\
 &+ \sum_{\ell' \in \calK_k } D_{\ell',m',\ell,m}(s_{k,\ell'})
  \frac{\Gamma^2_\ell(\tau_{k-1})}{ \Gamma^2_{\ell}(s_{k,\ell'})} (t-s_{k,\ell'}) \\
&\le \bbE( (\wtd{X}_{\ell} (\tau_{k-1}))^2) + 
     \sum_{ \ell' \in \calK_k} D_{\ell',m',\ell,m} (s_{k,\ell'})/ n_{\ell'}
\end{empheq}
\end{comment}
where
$
  D_{\ell',m',\ell,m}(t) := 
%\bbE( \langle T_g(\widehat{X}^L(t)) Y_{\ell',m'}, Y_{\ell,m}\rangle^2 ).
\bbE( \langle {B}(\widehat{X}^L(t)) Y_{\ell',m'}, Y_{\ell,m}\rangle^2 ).
$
%\sout{We note that (need to check!!!)}
%
%\begin{empheq}[box=\shadebox]{equation}\label{Dbounded}
%\sum_{\ell=0}^L \sum_{|m| \le \ell} D_{\ell', m ',\ell,m} (t) \preceq 1 + %\bbE\|\what{X}^L(t)\|^2
%\end{empheq}
%\sout{due to .....We therefore get}
Thus, by virtue of Lemma \ref{lem:bd D} we have
\[
\bbE\| \wtd{X}(t) - \wtd{X}(\tau_{k-1}) \|^2
\preceq 1 + \max_{j=0,\ldots,k-1} \bbE \|\wtd{X}^L(\tau_j) \|^2,
\]
and we conclude that
\[
  f(s) := \sup_{ r \in [0,s] } \bbE \|{\wtd{X}^L(r)} \|
\]
is finite for $s\in [0,1]$, since $\bbE \|\wtd{X}^L(0)\|^2 = \|\xi\|^2 <\infty$.

Similar to \eqref{equ:recur}, we have
\begin{align*}
 \wtd{X}_{\ell m} (t) &= \Gamma_j(t) \inprod{\xi}{Y_{\ell m}} \\
  & +\sum_{\ell'=0}^{\Lambda} \sum_{|m'|\le \ell'} 
    \sum_{t_{j,\ell'} \le \tau_m }
   {\sqrt{A_{\ell'}}} \inprod{B(\wtd{X}(t_{j-1,\ell'}) Y_{\ell' m'}}{Y_{\ell m}} 
      \frac{\Gamma_\ell(t)}{\Gamma_\ell(t_{j-1,\ell'}) } \\
  &  \cdot ( w_{\ell' m'}(t \wedge t_{j,\ell'}) - w_{\ell' m'}(t_{j-1,\ell'}) ),
\end{align*}
which implies
\begin{align*}
\bbE( (\wtd{X}_{\ell m})^2 ) &= \Gamma^2_\ell(t) \inprod{\xi}{Y_{\ell m}}^2 \\
 & + \sum_{\ell'=0}^{\Lambda} 
    \sum_{|m'| \le \ell'}
  \sum_{t_{j,\ell'} \le \tau_k}
      D_{\ell',m',\ell,m} 
      \frac{\Gamma^2_\ell(t)}{\Gamma^2_\ell(t_{j-1,\ell'}) }
      A_{\ell'} (t \wedge t_{j,\ell'} - t_{j-1,\ell'}).
\end{align*}
%Using \eqref{Dbounded} 
Applying Lemma~\ref{lem:bd D} again, 
we have
\begin{align*}
 \bbE \|{\wtd{X}^L(t)}\|^2 & \preceq \|\xi\|^2 + 
     \sum_{\ell'=0}^{\Lambda}
	     {\sum_{|m'|\le\ell'}}
		     A_{\ell'} 
		     \sum_{t_{j,\ell'} \le \tau_k}
		     (1 + f(t_{j-1,\ell'}))(t \wedge t_{j,\ell'} - t_{j-1,\ell'}) \\
  & \preceq 1 + \int_0^t f(s) \mathrm{d}s,
\end{align*}
and due to Gronwall's lemma we can conclude that 
\begin{equation}\label{Xtldbounded}
 \sup_{t \in [0,1] } \bbE \| \wtd{X}^L(t) \|^2 \preceq 1.
\end{equation}
For the process $\what{X}^L$ we apply \eqref{equ:recur} again
to get
\begin{equation}\label{equ:6 13}
\bbE((\what{X}^L_{\ell m})^2) = \Gamma_\ell^2(t) \inprod{\xi}{Y_{\ell m}}^2
 +
	\sum_{\ell'=0}^{\Lambda} 
		\sum_{|m'|\le \ell'}
			A_{\ell'}/n_{\ell'} 
        \sum_{ t_{j,\ell'} \le \tau_k }
	       D_{\ell',m',\ell,m}(t_{j-1,\ell'}) 
           \frac{\Gamma^2_\ell(t)}
            {\Gamma^2_\ell(t_{j-1,\ell'})}.
\end{equation}
Using \eqref{Xtldbounded} we conclude that
\[
\bbE \|\what{X}^L(t)\|^2 \preceq \| \xi \|^2 + 
      \sum_{\ell'=0}^\Lambda
       \sum_{|m'|\le\ell'}A_{\ell'}
      (1+ \max_{j=0,\ldots,n_{\ell'}} \bbE \| \wtd{X}^L(t_{j,\ell'}) \|^2)
     \preceq 1.
\]
\end{proof}
%%%%%%%%%%%%%%%%%%%%%%%%%%%%%%%%%%%%%%%%%%%%%%%
To proceed, we want a spatially-discrete counterpart of Lemma \ref{lem:Xcts}. 
It turns out our scheme is almost square-mean continuous, and the discontinuity is controlled by the discretization of the Wiener process.
\begin{lemma}\label{lem:Xhat_cts}
{For the fully discrete solution $\what{X}^L$ defined as in} \eqref{eqn:full discrete}, 
{we have}:
\[
\bbE\| \what{X}^L(s) - \what{X}^L(t)\|^2
\preceq (t-s) (1+ \what{\psi}(s)) + 
 \sum_{\ell'=0}^{\Lambda}
	 {\sum_{|m'|\le \ell'}}
		 \frac{A_{\ell'}}{n_{\ell'}},
\]
where
$
\what{\psi}(s) = \sum_{\ell=0}^L \sum_{|m'|\le \ell'} \mu_\ell \bbE [ (\what{X}^L_{\ell,m}(s))^2 ]
$. 
Moreover,
\begin{equation}\label{int bounded}
 \int_0^1 \what{\psi}(s) \mathrm{d}s \preceq 1.
\end{equation}
\end{lemma}
\begin{proof}
For each $\ell,m,\ell',m'$, we have
\begin{align}
&\sum_{k=1}^K
	\int_{\tau_{k-1}}^{\tau_k}
		\sum_{ t_{j,\ell'} \le \tau_k }
			D_{\ell',m',\ell,m}(t_{j-1,\ell'}) 
				\frac{\Gamma^2_\ell(s)}
					{\Gamma^2_\ell(t_{j-1,\ell'})} \mathrm{d}s
\\
&
%{\Bigg(
=
\sum_{ 0< t_{j,\ell'} \le \tau_1 }
	\int_{0}^{\tau_2}
			D_{\ell',m',\ell,m}(t_{j-1,\ell'}) 
				\frac{\Gamma^2_\ell(s)}
					{\Gamma^2_\ell(t_{j-1,\ell'})} \mathrm{d}s\\
&\quad
%{+
+
\sum_{ \tau_1 < t_{j,\ell'} \le \tau_2 }
	\int_{\tau_1}^{\tau_2}
			D_{\ell',m',\ell,m}(t_{j-1,\ell'}) 
				\frac{\Gamma^2_\ell(s)}
					{\Gamma^2_\ell(t_{j-1,\ell'})} \mathrm{d}s\\
&\quad
%{+
+
\sum_{k=3}^K
		\int_{\tau_{k-1}}^{\tau_k}
			\sum_{ t_{j,\ell'} \le \tau_k }
				D_{\ell',m',\ell,m}(t_{j-1,\ell'}) 
					\frac{\Gamma^2_\ell(s)}
						{\Gamma^2_\ell(t_{j-1,\ell'})} \mathrm{d}s \\
%\Bigg)}\\
&=
\sum_{k=1}^K
\sum_{ \tau_{k-1} < t_{j,\ell'} \le \tau_{k} }
	\int_{\tau_{k-1}}^{1}
			D_{\ell',m',\ell,m}(t_{j-1,\ell'}) 
				\frac{\Gamma^2_\ell(s)}
					{\Gamma^2_\ell(t_{j-1,\ell'})} \mathrm{d}s \\
&\le
\sum_{k=1}^K
\sum_{ \tau_{k-1} < t_{j,\ell'} \le \tau_{k} }
	\int_{t_{j-1,\ell'}}^{1}
			D_{\ell',m',\ell,m}(t_{j-1,\ell'}) 
				\frac{\Gamma^2_\ell(s)}
					{\Gamma^2_\ell(t_{j-1,\ell'})} \mathrm{d}s.
\end{align}

%\sout{Since $s \in (\tau_{k-1},\tau_k)$ and $t_{j,\ell'} \le \tau_k$ implies $t_{j-1,\ell'} \le s$} 
From Lemma~\ref{lem:semigroup} and \eqref{equ:6 13}, it follows that 
\begin{align*}
\int_0^1 \bbE ( (\what{X}^L_{\ell m}(s))^2)
&=\sum_{k=1}^K\int_{\tau_{k-1}}^{\tau_k}\bbE ( (\what{X}^L_{\ell m}(s))^2) \mathrm{d}s
  \\
  &\le \inprod{\xi}{Y_{\ell m}}^2 \int_0^1 \Gamma_j^2(s) \mathrm{d}s \\
  & +
	  \sum_{\ell'=0}^{\Lambda} 
		  \sum_{|m'|\le \ell'}
		    \frac{ A_{\ell'} }{ n_{\ell'} }
		        \sum_{ j=0 }^{ n_{\ell'}-1 }
		       D_{\ell',m',\ell,m}(t_{j-1,\ell'})
			       \int_{ t_{j-1,\ell'} }^1
		           \frac{\Gamma^2_\ell(s)}
		            {\Gamma^2_\ell(t_{j-1,\ell'})} \mathrm{d}s \\
  & \preceq \frac{1}{\mu_\ell}
      \left( \inprod{\xi}{Y_{\ell m}}^2
   +
	   \sum_{\ell'=0}^{\Lambda}
		  \sum_{|m'|\le \ell'} \frac{A_{\ell'}}{n_{\ell'}}
		   \sum_{j=0}^{n_{\ell'}-1} D_{\ell', m', \ell,m}(t_{j,\ell'})\right).
\end{align*}
It follows that
\[
\int_0^1 \what{\psi}(s) \mathrm{d}s \preceq \|\xi\|^2 + 
\sum_{\ell'=0}^{\Lambda}
{\sum_{|m'|\le \ell'}}
		\frac{A_{\ell'}}{ n_{\ell'}} \sum_{j=0}^{n_{\ell'}-1}
       (1 + \bbE\|{\what{X}^L(t_{j,\ell'})} \|^2).
\]
From Lemma~\ref{lem:bounded}, we obtain \eqref{int bounded}.

Assume that $s<t$ with $s\in[\tau_{k-1},\tau_k]$ and $t \in (\tau_{\zeta-1},\tau_\zeta]$ for
$k \le \zeta$. Then
\begin{align*}
&\bbE ( \what{X}_{\ell m}^L(s) - \what{X}_{\ell m}^L(t) )^2 \\
&\quad=
\left(1 - \frac{\Gamma_j(t)}{\Gamma_j(s)}  \right)^2
 \bbE(\what{Y}_{\ell m}(s))^2
 + \sum_{\ell'=0}^{\Lambda}
 \sum_{|m'|\le \ell'}
  \frac{A_{\ell'}}{n_{\ell'}} 
  \sum_{j \in \calK_{\ell'}(s,t)}
  D_{\ell',m',\ell,m}(t_{j-1,\ell'})
  \frac{\Gamma^2_\ell(t)}{\Gamma^2_{\ell}(t_{j-1,\ell'})},
\end{align*}
where
\[
\calK_{\ell'}(s,t) = \{ j \in \{1,\ldots,n_{\ell'}\} : t_{j,\ell'} \in (\tau_k, \tau_{\zeta}] \}\qquad
\text{if $s > \tau_{k-1}$,}
\]
and
\[
\calK_{\ell'}(s,t) = \{ j \in \{1,\ldots,n_{\ell'}\} : t_{j,\ell'} \in [\tau_k, \tau_{\zeta}] \}
\qquad
\text{if $s = \tau_{k-1}$.}
\] 
By Lemma~\ref{lem:semigroup}, we have
\[
 \bbE(\what{X}_{\ell m}(s) - \what{X}_{\ell m}(t))^2
 \preceq \mu_\ell (t-s) \bbE( (\what{X}_{\ell m}(s))^2 )
  + \sum_{\ell'=0}^\Lambda
	  \sum_{|m'|\le \ell'}
		  \frac{A_\ell'}{n_{\ell'}} 
       \sum_{j \in \calK_{\ell'}(s,t)} 
          D_{\ell',m',\ell,m}(t_{j-1,\ell'}).
\]
Note that $\#\calK_{\ell'}(s,t) \le 1 + n_{\ell'}(t-s)$, then use %\sout{\eqref{Dbounded} together with}
Lemma~\ref{lem:bd D} and \ref{lem:bounded} to obtain
\begin{align*}
\bbE\| \what{X}^{{L}}(s) - \what{X}^{{L}}(t)\|^2
 &\preceq (t-s) \what{\psi}(s) +
 \sum_{\ell'=0}^{\Lambda}
 \sum_{|m'|\le \ell'}
	 \frac{A_{\ell'}}{n_{\ell'}} \#\calK_{\ell'}(s,t)\\
 &\preceq (t-s) (1 + \what{\psi}(s)) + \sum_{\ell'=0}^{\Lambda}
 \sum_{|m'|\le \ell'}
 \frac{A_{\ell'}}{n_{\ell'}}.
\end{align*}
\end{proof}
%%%%%%%%%%%%%%%%%%%%%%%%%%%%%%%%%%%%%%%%%%%%%%%%%
\begin{comment}
In view of Lemmata~\ref{lem:bounded} and \ref{lem:Xhat_cts}, we may proceed as in the proof of 
%Theorem~\ref{thm:semidiscrete} 
{Lemma~\ref{lem:piecewise}}
to obtain the following error bound for piecewise constant interpolation of $\what{X}^L$.
\end{comment}
%\footnote{{It seems that Lemma~\ref{lem:sumXhatL} directly (Lemma \ref{lem:bounded} is used in \ref{lem:Xhat_cts} but) follows from Lemma \ref{lem:Xhat_cts}. I think the proof is correct? But then the by   Gronbach-M\"{u}ller--Ritter, BIT}
%\begin{quote}
%In view of Lemmata 6.4 (continuity result with the function $\psi$) and 6.5 (boundedness of the moment) we may proceed as in the proof of Lemma 4.1 (there is no Lemma 4.1, it should be Proposition 4.1)
%to obtain the following error bound for piecewise constant interpolation of $\hat{X}$.
%\end{quote} is even more confusing. }
{We need the following error bound for piecewise constant interpolation of $\what{X}^L$ to show our main result.}
\begin{lemma}\label{lem:sumXhatL}
\[
\sum_{j=1}^{n_{\ell'}}
\int_{t_{j-1,\ell'}}^{t_{j,\ell'}} \bbE \| \what{X}^L(t) - \what{X}^L(t_{j-1,\ell'})\|^2 \mathrm{d}t
\preceq \frac{1}{n_{\ell'}} +
\sum_{\lambda'=0}^\Lambda
		\sum_{|\mu'|\le \lambda'}
			\frac{A_{\lambda'}}{n_{\lambda'}}.
\]
\end{lemma}
\begin{proof}
%\footnote{{I gave a proof but it is very short.}}
Lemma~\ref{lem:Xhat_cts} implies 
\begin{align*}
\sum_{j=1}^{n_{\ell'}}
\int_{t_{j-1,\ell'}}^{t_{j,\ell'}}& \bbE \| \what{X}^L(t) - \what{X}^L(t_{j-1,\ell'})\|^2 \mathrm{d}t\\
&\preceq 
\sum_{\ell'=0}^{\Lambda}
		 {\sum_{|m'|\le \ell'}}
			 \frac{A_{\ell'}}{n_{\ell'}} + 
\sum_{j=1}^{n_{\ell'}}
\int_{t_{j-1,\ell'}}^{t_{j,\ell'}} 
	\bigg[
		\bigg(\sup_{s\in(t_{j-1,\ell'},t_{j,\ell'})}|s-t_{j-1,\ell'}|\bigg)(1+\what{\psi}(t))
	\bigg]\mathrm{d}t\\
&\preceq 
	\sum_{\ell'=0}^{\Lambda}
		 {\sum_{|m'|\le \ell'}}
			 \frac{A_{\ell'}}{n_{\ell'}} +
	\frac1{n_{\ell'}}\Big(1+\int_{0}^{1}\what{\psi}(t)\mathrm{d}t\Big)
\preceq 	 \sum_{\ell'=0}^{\Lambda}
		 {\sum_{|m'|\le \ell'}}
			 \frac{A_{\ell'}}{n_{\ell'}} + \frac1{n_{\ell'}} .
\end{align*}
\end{proof}
We are ready to state our main result.
\begin{theorem}\label{thm:main result}
{The fully discrete solution defined in \eqref{eqn:full discrete} satisfies the following error estimate}
\[
\bbE\left(  \int_0^1 \|X(t) - \what{X}^L(t) \|^2 \mathrm{d}t \right)
\preceq \frac{1}{L^2} + 
  \sum_{\ell'=0}^{\Lambda}
	 \sum_{|m'|\le \ell'}
	  \frac{A_{\ell'}}{n_{\ell'}} + 
\sum_{\ell'>\Lambda}
	\sum_{|m'|\le \ell'}
		A_{\ell'}.
\]
\end{theorem}
\begin{proof}
%\footnote{(See M\"{u}ller-Gronbach--Ritter (2007), BIT, pp. 415--)}
In view of Theorem~\ref{thm:semidiscrete}, it suffices to show
that
\begin{equation}\label{equ:key}
\int_0^1 \bbE \| X^L(t) - \what{X}^L(t)\|^2 \mathrm{d}t \preceq 
\sum_{\ell'=0}^\Lambda
\sum_{|m'|\le \ell'}\frac{A_{\ell'}}{n_{\ell'}}
.
\end{equation}
For $\nu=1,2,3$ we define
\[
 U^{(\nu)}_{\ell,m}(t)
 = \sum_{\ell'=0}^{\Lambda} \sum_{|m'| \le \ell'}
	\sqrt{A_{\ell'}}
 \int_0^t \sum_{j=0}^{n_{\ell'}-1} 
 V^{(\nu)}_{{\ell',m',\ell,m},j}(s,t) 1_{(t_{j,\ell'}, t_{j+1,\ell'}]}(s) \mathrm{d} w_{\ell',m'}(s)
\]
with
\begin{align*}
V^{(1)}_{{\ell',m',\ell,m,j}}(s,t) 
&=\exp(-\mu_\ell(t-s))
\inprod{({B}(X^L(s)) - {B}(X^L(t_{j,\ell'}))) Y_{{\ell' m'}}} {Y_{{\ell m}}},\\
V^{(2)}_{{\ell',m',\ell,m,j}}(s,t) 
&=\exp(-\mu_\ell(t-s))
\inprod{({B}(X^L({t_{j,\ell'}})) - {B}(\what{X}^L(t_{j,\ell'}))) Y_{{\ell' m'}}} {Y_{{\ell m}}},\\
V^{(3)}_{{\ell',m',\ell,m,j}}(s,t) 
&= 
\left( \exp(-\mu_\ell(t-s)) - \frac{\Gamma_\ell(t)}{\Gamma_\ell(t_{j,\ell'})} \right)
\inprod{{B}(\what{X}^L(t_{j,\ell'})) Y_{{\ell' m'}}} {Y_{{\ell m}}}.
\end{align*}
For $t \in (\tau_{k-1},\tau_k]$. Let
\[
{U^{(4)}_{\ell,m}}(t) = 
\sum_{\ell'\in \{0,\dotsc,\Lambda\}\setminus\calK_k} \sum_{|m'| \le \ell'}
		{ \sqrt{A_{\ell'}} }
		\frac{\Gamma_\ell(t)}{\Gamma_{\ell}( {s_{k,\ell'}} )} 
		\inprod{{B}(\what{X}^L( {s_{k,\ell'}} )) Y_{{\ell' m'}}} {Y_{{\ell m}}}
		( { w_{\ell'm'}(t) - w_{\ell'm'}( s_{k,\ell'} )} ),
\]
and
\[
{U^{(5)}_{\ell,m}(t)} = 
\sum_{\ell'\in \calK_k} \sum_{|m'| \le \ell'}
	{ \sqrt{A_{\ell'}} }
	\frac{\Gamma_\ell(t)}{\Gamma_{\ell}( {s_{k,\ell'}} )} 
	\inprod{{B}(\what{X}^L( {s_{k,\ell'}} )) Y_{{\ell' m'}}} {Y_{{\ell m}}}
	( w_{\ell'm'}( {\tau_k} ) - w_{\ell'm'}(t)).
\]
Then, by definition
\begin{align*}
X^L_{\ell,m}(t) - \widehat{X}^L_{\ell,m}(t)
 &= (\exp(-\mu_\ell t)- \Gamma_\ell(t)) \cdot \inprod{\xi}{Y_{\ell m}} \\
 &\quad + U^{(1)}_{\ell,m}(t) + U^{(2)}_{\ell,m}(t) + U^{(3)}_{\ell,m}(t) + U^{(4)}_{\ell,m}(t)
 {- U^{(5)}_{\ell,m}(t)}.
\end{align*}
We will estimate each term separately using results in previously stated lemmas.

Using Lemma~\ref{lem:semigroup}, we have
\[
\sum_{\ell=0}^{L} \sum_{|m|\le \ell}
 \inprod{\xi}{Y_{\ell,m}}^2 \int_0^1 
 (\exp(-\mu_\ell t) - \Gamma_\ell (t))^2 \mathrm{d}t
 \preceq 1/n^* \preceq
 \sum_{\ell'\le \Lambda}
	 {\sum_{|m'|\le \ell'}}
		 A_{\ell'}/n_{\ell'}.
\]

%Using the {Lipschitz} condition and Lemma~\ref{lem:piecewise} we obtain
{Letting $f=T_g(X)-T_g(X^{L})$ in Lemma~\ref{lem:Ylpmplm id}, from }\eqref{equ:Tgu H Lip} {and $\sup_{\ell'm'}|\eta_{\ell'm'}|<\infty$, together with} 
Lemma~\ref{lem:piecewise} we obtain
\begin{align*}
\sum_{\ell=0}^{L} \sum_{|m| \le \ell} \bbE (U^{(1)}_{\ell,m}(t))^2
 &\preceq \sum_{\ell'=0}^{\Lambda} \sum_{|m'|\le \ell'}
 {A_{\ell'}}
  \sum_{j=0}^{{n_{\ell'}}-1} 
 \int_{ {t_{j,\ell'}} }^{ {t_{j+1,\ell'}} }  \bbE\| X^L(s) - X^L({t_{j,\ell'}}) \|^2 \mathrm{d}s \\
 &\preceq \sum_{\ell' \le \Lambda}
						 {\sum_{|m'|\le \ell'}}
								A_{\ell'}/n_{\ell'}.
\end{align*}

Put
\[
{h}(s) = \bbE \| X^L(s) - \widehat{X}^L(s)\|^2,
\]
which is finite because of Lemma~\ref{lem:bounded} and \eqref{XLbounded}.
By the linear growth condition, Lemma~\ref{lem:piecewise} and Lemma~\ref{lem:sumXhatL},
we have
\begin{align*}
\sum_{\ell \le L}\sum_{|m|\le \ell} 
\bbE(U^{(2)}_{\ell,m}(t))^2
&\preceq \sum_{\ell' \le \Lambda} \sum_{|m'|\le \ell'}
 { A_{\ell'} }
\sum_{j=0}^{ {n_{\ell'}-1} } 
\int_{t \wedge {t_{j,\ell'}}}^{ t \wedge {t_{j+1,\ell'}} }
(\bbE\| X^L(s) - X^L( {t_{j,\ell'}} )\|^2 + \\
&\qquad \qquad \qquad \qquad \bbE\|\widehat{X}^L(s)-\widehat{X}^L( {t_{j,\ell'}})\|^2 + {h}(s)) \mathrm{d}s \\
& \preceq \sum_{\ell' \le \Lambda} {\sum_{|m'|\le \ell'}} \frac{A_{\ell'}}{n_{\ell'}} + \int_0^t {h}(s) \mathrm{d}s.
\end{align*}
Next, we estimate $\int_0^1 \bbE(U_{\ell,m}^{(3)}(t))^2 \mathrm{d}t $. 
Suppose that {$s\in (t_{j,\ell'}, t_{j+1,\ell'}]$}. Then
\[
|\exp(-\mu_\ell(t-s)) - \exp(-\mu_\ell(t-{t_{j,\ell'}})) | \le  \exp(-\mu_\ell(t-s))\mu_\ell/n_{{\ell'}}.
\]
Therefore,
\begin{align}
\begin{split}
&\int_s^1
\left(
  \exp(-\mu_{\ell} (t-s))
   -\frac{ \Gamma_{\ell}(t) }{ \Gamma_{\ell}{ (t_{j,\ell'}) } }
\right)^2 \mathrm{d}t\\
&=
\int_s^1
\left(
  \exp(-\mu_{\ell} (t-s))
  -\exp(-\mu_{\ell} (t-t_{j,\ell'}))
  +\exp(-\mu_{\ell} (t-t_{j,\ell'}))
   -\frac{ \Gamma_{\ell}(t) }{ \Gamma_{\ell}{ (t_{j,\ell'}) } }
\right)^2 \mathrm{d}t \\
&\le
2
\int_s^1
|
  \exp(-\mu_{\ell} (t-s))
  -\exp(-\mu_{\ell} (t-t_{j,\ell'}))|^2
\mathrm{d}t
  +
2\int_s^1
\left(
\exp(-\mu_{\ell} (t-t_{j,\ell'}))
   -\frac{ \Gamma_{\ell}(t) }{ \Gamma_{\ell}{ (t_{j,\ell'}) } }
\right)^2 \mathrm{d}t.\\
&
\le 2\frac{ \mu_{\ell}^2 }{n_{\ell'}^2 }
\int_s^1
  \exp(-2\mu_{\ell} (t-s))
\mathrm{d}t
  +
2\int_{t_{j,\ell'}}^1
\left(
\exp(-\mu_{\ell} (t-t_{j,\ell'}))
   -\frac{ \Gamma_{\ell}(t) }{ \Gamma_{\ell}{ (t_{j,\ell'}) } }
\right)^2 \mathrm{d}t.
\label{equ:BIT p. 417}
\end{split}
\end{align}
Consider the case $\mu_{\ell}/n_{\ell'}\le 1$.  For the integral in the first term of \eqref{equ:BIT p. 417} we have
\begin{align}
\int_s^1
  \exp(-2\mu_{\ell} (t-s))
\mathrm{d}t
&=
\frac{1}{-2\mu_{\ell}}e^{-2\mu_{\ell}(1-s)} - \frac{1}{-2\mu_{\ell}}e^0=\frac1{2\mu_{\ell}}(1-e^{-2\mu_{\ell}(1-s)})\\
&\le
\min\{ \frac1{2\mu_{\ell}},\, \frac{2\mu_{\ell}}{2\mu_{\ell}}(1-s) \} \le \frac1{2\mu_{\ell}},
\end{align}
which gives us the estimate for the first term
$$
2\frac{ \mu_{\ell}^2 }{n_{\ell'}^2 }\frac1{2\mu_{\ell}}
=
\frac{ \mu_{\ell} }{n_{\ell'} }\frac{ 1 }{n_{\ell'} }
\le \frac{ 1 }{n_{\ell'} }.
$$
From Lemma \eqref{lem:semigroup}, the second term can be bounded by $\preceq { 1 }/{n_{\ell'} }$. 

If $\mu_{\ell}/n_{\ell'}\ge 1$, then we have $e^{-\mu_{\ell}x} \le e^{-n_{\ell'}x} $ ($x\ge 0$). Thus, \eqref{equ:BIT p. 417} can be bounded by
\begin{align}
&4
\int_s^1
|
  \exp(-\mu_{\ell} (t-s))|^2
  \mathrm{d}t + 4\int_{t_{j,\ell'}}^1\exp(-\mu_{\ell} (t-t_{j,\ell'}))|^2
\mathrm{d}t \\
&
4
\int_s^1
|
  \exp(-n_{\ell'} (t-s))|^2
  \mathrm{d}t + 4\int_{t_{j,\ell'}}^1|\exp(-n_{\ell'} (t-t_{j,\ell'}))|^2
\mathrm{d}t \\
&\le
4\frac1{2 n_{\ell'}} + 4\frac1{2 n_{\ell'}}.
\end{align}
Therefore,
\begin{align*}
&\int_s^1 \left( \exp(-\mu_\ell(t-s)) - \frac{\Gamma_\ell(t)}{\Gamma_\ell(t_{j,\ell'})} \right)^2 \mathrm{d}t 
%& \qquad \preceq \frac{1}{{n_{\ell'}}} + 
% \int_{t_{j,\ell'}}^1  
% \left( \exp(-\mu_\ell(t-t_{j,\ell'})) - \frac{\Gamma_\ell(t)}{\Gamma_\ell(t_{j,\ell'})} \right)^2 \mathrm{d}t \\
%& \qquad
\preceq
\frac{1}{n_{\ell'}}.
\end{align*}
%follows from Lemma~\ref{lem:semigroup}. 
Thus, with
$
  D_{\ell',m',\ell,m}(t) = 
\bbE( \langle {B}(\widehat{X}^L(t)) Y_{\ell',m'}, Y_{\ell,m}\rangle^2 ) 
$ 
we have
\begin{align*}
\int_0^1 \bbE(U_{\ell,m}^{(3)}(t))^2 \mathrm{d}t 
  &\preceq  \sum_{\ell'=0}^{\Lambda}
      \sum_{|m'| \le \ell'}
      { A_{\ell'}}
    \int_0^1\!\!\! \int_0^t \sum_{j=0}^{{n_{\ell'}-1}}
\left( \exp(-\mu_\ell(t-s)) -  \frac{\Gamma_\ell(t)}{\Gamma_\ell({t_{j,\ell'}})}\right)^2\\
 & \qquad \qquad \qquad \qquad \times  D_{\ell',m',\ell,m}({t_{j,\ell'}}) 1_{{ (t_{j,\ell'},t_{j+1,\ell'}] }} (s) \mathrm{d}s{\mathrm{d}t}  \\
&{=\sum_{\ell'=0}^{\Lambda}
       \sum_{|m'| \le \ell'}
        A_{\ell'}
     \int_0^1\!\!\! \int_s^1 \sum_{j=0}^{n_{\ell'}-1}
 \left( \exp(-\mu_\ell(t-s)) -  \frac{\Gamma_\ell(t)}{\Gamma_\ell({t_{j,\ell'}})}\right)^2}\\
  &\qquad \qquad \qquad \qquad \times  D_{\ell',m',\ell,m}({t_{j,\ell'}}) 1_{{ (t_{j,\ell'},t_{j+1,\ell'}] }} (s) \mathrm{d}t\,ds  \\
  &\preceq\sum_{\ell'=0}^{\Lambda}
         \sum_{|m'| \le \ell'}
          A_{\ell'}
       \int_0^1
        \sum_{j=0}^{n_{\ell'}-1}
			\frac1{n_{\ell'}} D_{\ell',m',\ell,m}({t_{j,\ell'}}) 1_{{ (t_{j,\ell'},t_{j+1,\ell'}] }} (s) \mathrm{d}s \\
&
=\sum_{\ell'=0}^{\Lambda}
       \frac{ A_{\ell'}}{n_{\ell'}^2}
        \sum_{j=0}^{n_{\ell'}-1}
        \sum_{|m'| \le \ell'}
		 D_{\ell',m',\ell,m}({t_{j,\ell'}}).
\end{align*}
From the Lemma \ref{lem:Ylpmplm id}, for any $\ell'\in\{1,\dots,\Lambda\}$ we have
\begin{align}
\sumlm
\sum_{|m'| \le \ell'}D_{\ell',m',\ell,m}(t)\le
{\big(\sup_{\lambda,\nu}|\eta_{\lambda,\nu}|\big)^2}
\frac{2\ell'+1}{4\pi}
{\|T_g(\widehat{X}^L(t)) \|^2}
,
\end{align}
and thus Lemma~\ref{lem:bounded} implies
\begin{equation}
\sum_{\ell=0}^L \sum_{|m|\le \ell} 
\int_0^1 \bbE(U_{\ell,m}^{(3)} (t))^2 \mathrm{d}t 
\preceq
	\sum_{\ell'\le \Lambda}
		\frac{2\ell'+1}{4\pi}
		\frac{A_{\ell'}}{n_{\ell'}}
\asymp
	\sum_{\ell'\le \Lambda}
		\sum_{|m'|\le \ell'}
		\frac{A_{\ell'}}{n_{\ell'}}.
\end{equation}
The same facts {yield}
\begin{equation}
\sum_{\ell \le L} \sum_{|m|\le \ell}
\bbE (U^{(4)}_{\ell,m}(t))^2
\le 
	\sum_{\ell'\le \Lambda} 
		\sum_{|m'|\le\ell'}
		\frac{A_{\ell'}}{n_{\ell'}},
\end{equation}
and
\begin{equation}
\sum_{\ell \le L} \sum_{|m|\le \ell}
\bbE (U^{(5)}_{\ell,m}(t))^2
\le 
	\sum_{\ell'\le \Lambda} 
		\sum_{|m'|\le\ell'}
		\frac{A_{\ell'}}{n_{\ell'}}.
\end{equation}
Combining above estimates, we obtain
\[
\int_0^r {h}(t) \mathrm{d}t \preceq \sum_{\ell' \le \Lambda}
	\sum_{|m'|\le\ell'} \frac{A_{\ell'}}{n_{\ell'}}
 + \int_0^r \int_0^t {h}(s) \mathrm{d}s \mathrm{d}t.
\]
Finally, we apply Gronwall's lemma to derive
$
 \int_0^1 {h}(t) \mathrm{d}t \preceq \sum_{\ell'=0}^{\Lambda} \sum_{|m'|\le\ell'} \frac{A_\ell'}{n_{\ell'}},
$ 
as claimed in \eqref{equ:key}.
\end{proof}	
%%%%%%%%%%%%%%%%%%%%%%%%%%%%%%%%%%%%%%%%%%%%%%%%%%%%%%%%%%%%%%%
\section{Numerical experiments}\label{sec:num}
In this section, we consider the following equation
\begin{align*}
\begin{aligned}
\mathrm{d}X(t) &= \Delta^{\ast} X(t) \mathrm{d}t + X(t) \mathrm{d}W(t) \\
X(0)  &= \xi, \quad t\in [0,1]
\end{aligned}
\end{align*}
where $\xi$ {is} defined by
\[
   \xi = \sum_{\ell=0}^{100} \sum_{|m|\le \ell} \xi_{\ell,m} Y_{\ell,m},
\]
%{The original Figures are commented out for now since I do not have jpg data. }
\begin{comment}
% commented out since I don't have jpg fiels
\begin{center}
\begin{figure}[ht]
\includegraphics[scale=0.40]{N100_T0.jpg}
%\caption{Solution at $t=0$}
\includegraphics[scale=0.40]{N100_T500.jpg}
%\caption{Solution at $t=0.5$}
\includegraphics[scale=0.40]{N100_T1000.jpg}
\caption{Solution at $t=0; 0.5; 1.0$}
\end{figure}
\end{center}

\begin{center}
\includegraphics[scale=0.40]{AvgE2_N40.jpg}
\end{center}
\end{comment}
%
\begin{figure}[htbp]%[H]
  \centering
  \includegraphics[width=0.6\textwidth]{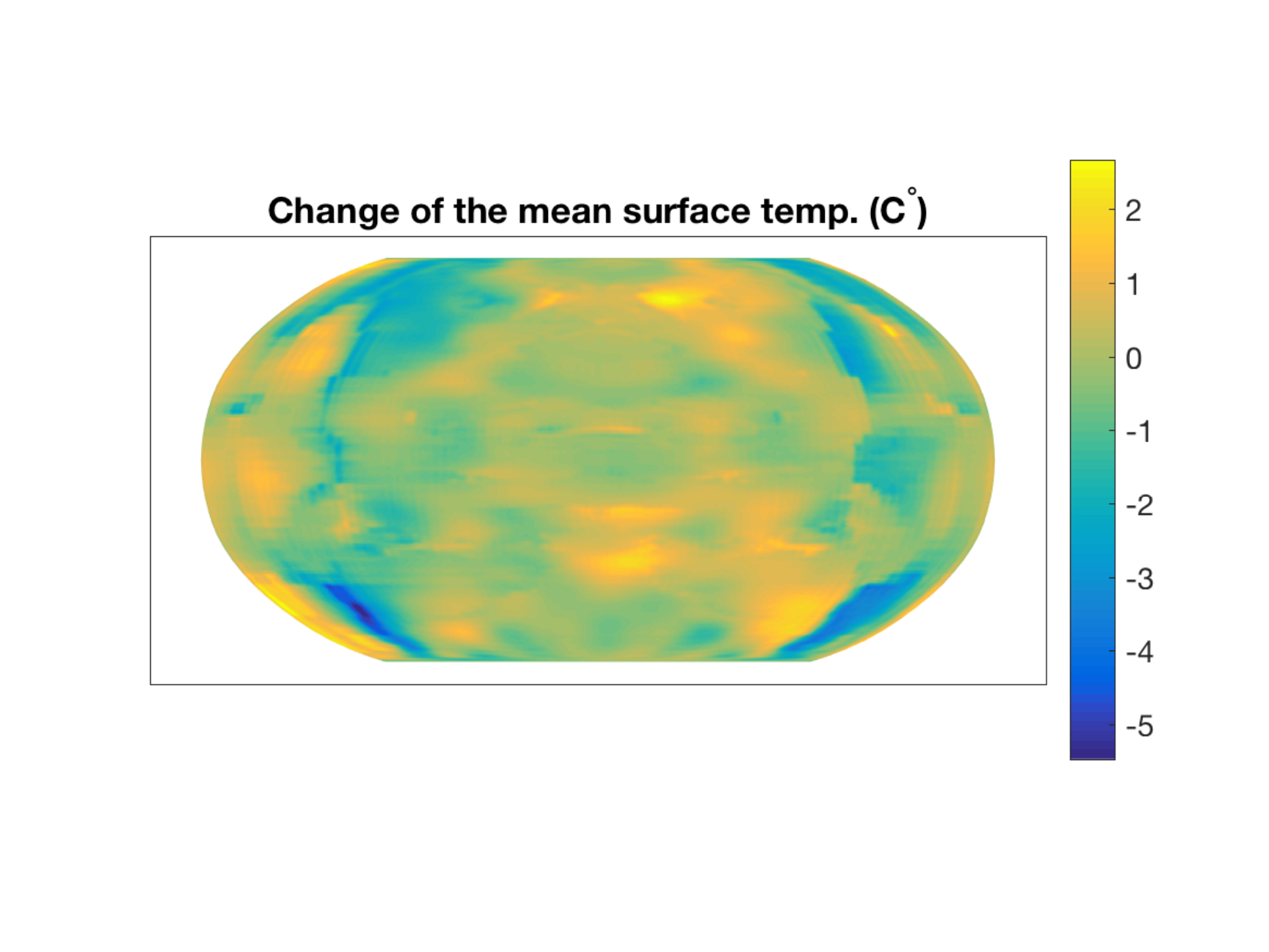}
  \caption{{Plot of the change of the mean surface temperature in June, from 2006 to 2016, approximated by the spherical harmonics up to degree $120$.}}
\label{fig:init plot}
\end{figure}
%\begin{comment}
\begin{figure}[htbp]%[H]
  \centering
  \includegraphics[width=0.8\textwidth]{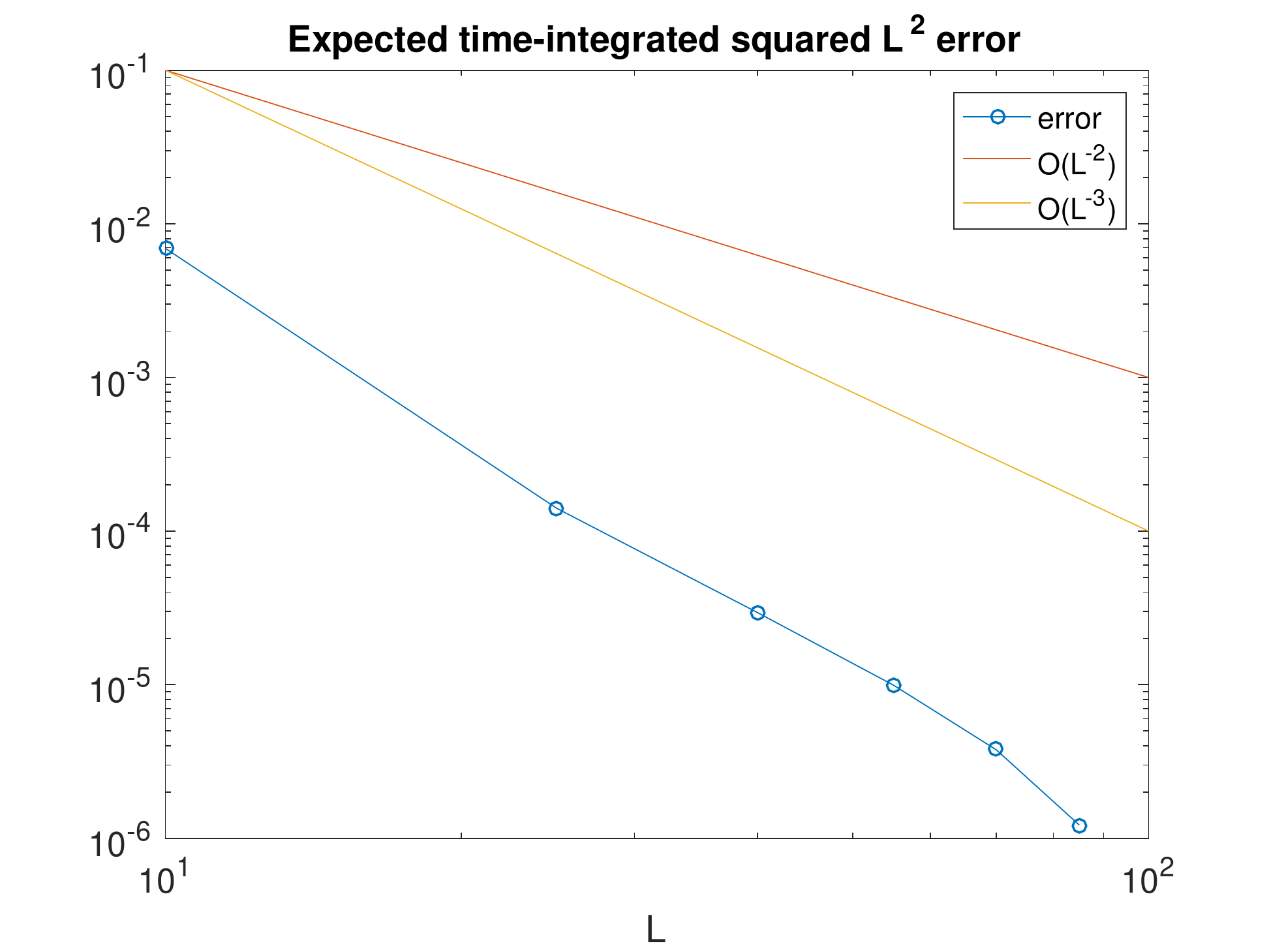}
  \caption{Plot of $\bbE\big(  \int_0^1 \|X(t) - \what{X}^L(t) \|^2 \mathrm{d}t \big)$. }
\label{fig:spatial trc}
\end{figure}
%\end{comment}
where $\xi_{\ell,m}$ are coefficients taken from 
GISTEMP Surface Temperature Analysis by NASA Goddard Institute for Space Sciences
(\url{http://data.giss.nasa.gov/gistemp/maps/}). 
The data describes the change of the mean surface temperature in June, from 2006 to 2016. See Figure \ref{fig:init plot} for a plot of the data, approximated by the spherical harmonics up to degree $120$.

\YK{Note that the mapping $B$ defined by the Nemytskii operator with a linear function together with point wise multiplication as in $B(x)f:=x\cdot f$ for $x\in H$, $f\in H_0$ satisfies the  condition \eqref{equ:Lip} and \eqref{equ:lin growth}.} 
We assume the $Q$-Wiener process $W$ on $H$ is defined by the covariance operator $Q$ such that 
$$
Q Y_{0} = 100,\quad\text{and} \quad
Q Y_{\ell' m'} = \frac{100}{{\ell'}^2}1_{\{\ell'\le10\}}(\ell')\quad \text{ for }\ell'\in\mathbb{N}.
$$
First, we consider the spatial truncation error. 
We choose $\frac1{n_{\ell'}}=\frac1n=0.004$ for $\ell'=1,\dotsc,10$, and consider 
$L=10,25,40,55,\dotsc,100$, where the case $L=100$ we see as a reference solution $X$. 
For each sample, $\int_0^1 \|X(t) - \what{X}^L(t) \|^2 \mathrm{d}t$ is approximated by 
$
\sum_{j=1}^N\|X(j/n) - \what{X}^L(j/n) \|^2\frac1n
$, and the expected value $\bbE\sum_{j=1}^N\|X(j/n) - \what{X}^L(j/n) \|^2\frac1n$ is approximated by Monte Carlo method with $100$ samples. 
Figure \ref{fig:spatial trc} shows the error decay of the second to third order, which is consistent with Theorem \ref{thm:main result}.
%
%
\begin{comment}
\begin{figure}[htbp]%[H]
  \centering
  \includegraphics[width=0.8\textwidth]{truncdt_Lmax100_Lambda10}
  \caption{Plot of $\bbE\big(  \int_0^1 \|X(t) - \what{X}^L(t) \|^2 \mathrm{d}t \big)$. }
\label{fig:temporal trc}
\end{figure}
Next, we consider the time discretisation. We choose $L=100$, and consider the time step $0.04,0.036,0.032,\dots,0.0120$. The computation of the error $\bbE(\int_0^1 \|X(t) - \what{X}^L(t) \|^2 \mathrm{d}t)$ is the same as above. 
\end{comment}

\section{Acknowledgement}
The authors are grateful to Ian H. Sloan, Klaus Ritter, 
	Thomas M\"{u}ller-Gronbach and Christoph Schwab for helpful conversations on the mathematical content of the paper. This work was undertaken with the assistance of computational resources from the UNSW HPC Service leveraging the National Computational Infrastructure (NCI), which is supported by the Australian Government, and also with the support from Australian Research Council's Discovery Project DP150101770.
\bibliographystyle{elsarticle-num} 


%\bibliographystyle{plain}
%\bibliography{mybib}
\bibliography{KL_heat_els.bib}

\begin{thebibliography}{10}
\expandafter\ifx\csname url\endcsname\relax
  \def\url#1{\texttt{#1}}\fi
\expandafter\ifx\csname urlprefix\endcsname\relax\def\urlprefix{URL }\fi
\expandafter\ifx\csname href\endcsname\relax
  \def\href#1#2{#2} \def\path#1{#1}\fi

\bibitem{DaPZab14}
G.~D. Prato, J.~Zabczyk, Stochastic Equations in Infinite Dimensions, 2nd
  Edition, Cambridge University Press, Cambridge, 2014.

\bibitem{GreKlo96}
W.~Grecksch, P.~E. Kloeden, Time-discretised {G}alerkin approximations of
  parabolic stochastic {PDE}s, Bull. Austral. Math. Soc. 54 (1996) 79--85.

\bibitem{GyoNua97}
I.~Gy\"{o}ngy, D.~Nualart, Implicit scheme for stochastic parabolic partial
  differential equation driven by space-time white noise, Potential Anal. 7
  (1997) 725--757.

\bibitem{AllNovZha98}
E.~J. Allen, S.~J. Novosel, Z.~Zhang, Finite element and difference
  approximation of some linear stochastic partial differential equations,
  Stoch. Stoch. Rep. 64 (1998) 117--142.

\bibitem{Gyo99}
I.~Gy\"{o}ngy, Lattice approximations for stochastic quasi-linear parabolic
  partial differential equations driven by space-time white noise {II},
  Potential Anal. 11 (1999) 1--37.

\bibitem{Sha99}
T.~Shardlow, Numerical methods for stochastic parabolic {{PDE}}s, Numer. Funct.
  Anal. Optim. 20 (1999) 121--145.

\bibitem{DavGai01}
A.~M. Davie, J.~Gaines, Convergence of numerical schemes for the solution of
  parabolic partial differential equations, Math. Comp. 70 (2001) 121--134.

\bibitem{DuZha02}
Q.~Du, T.~Zhang, Numerical approximation of some linear stochastic partial
  differential equations driven by special additive noises, SIAM J. Numer.
  Anal. 40 (2002) 1421--1445.

\bibitem{KloSho01}
P.~E. Kloeden, S.~Shott, Linear-implicit strong schemes for {I}to--{G}alerkin
  approximations of stochastic {PDE}s, J. Appl. Math. Stochastic Anal. 14
  (2001) 47--53.

\bibitem{Hau02}
E.~Hausenblas, Numerical analysis of semilinear stochastic evolution equations
  in {B}anach spaces, J. Comput. Appl. Math. 147 (2002) 485--516.

\bibitem{Hau03}
E.~Hausenblas, Approximation for semilinear stochastic evolution equations,
  Potential Anal. 18 (2003) 141--186.

\bibitem{LorRou04}
G.~J. Lord, J.~Rougemont, A numerical scheme for stochastic {PDE}s with gevrey
  regularity, IMA J. Num. Anal. 4 (2004) 587--604.

\bibitem{Yan04}
Y.~Yan, Semidiscrete {G}alerkin approximation for a linear stochastic,
  parabolic partial differential equation driven by additive noise, BIT Numer.
  Math. 44 (2004) 829--847.

\bibitem{Yan05}
Y.~Yan, Finite element methods for stochastic parabolic partial differential
  equations, SIAM J. Num. Anal. 43 (2005) 1363--1384.

\bibitem{GroRit07}
T.~M\"{u}ller-Gronbach, K.~Ritter, Lower bounds and nonuniform time
  discretization for approximation of stochastic heat equation, Found. Comput.
  Math. (2007) 135--181.

\bibitem{GroRit07a}
T.~M\"{u}ller-Gronbach, K.~Ritter, An implicit euler scheme with non-uniform
  time discretization for heat equations with multiplicative noise, BIT Numer.
  Math. 47 (2007) 393--418.

\bibitem{LanSch14}
A.~Lang, C.~Schwab, {Isotropic {G}aussian random fields on the sphere:
  Regularity, fast simulation and stochastic partial differential equations},
  Ann. Appl. Probab. 25~(6) (2015) 3047--3094.

\bibitem{Mul66}
C.~M\"{u}ller, Spherical Harmonics, Vol.~17 of Lecture Notes in Mathematics,
  Springer-Verlag, Berlin, 1966.

\bibitem{MarPec11}
D.~Marinucci, G.~Peccati, Random fields on the sphere. Representation, limit
  theorems and cosmological applications, Cambridge University Press,
  Cambridge, 2011.

\bibitem{NouMuiRai03}
T.~Nousiainen, K.~Muinonen, P.~R\"ais\"anen, Scattering of light by large
  {S}aharan dust particles in a modified ray optics approximation, J. Geophys.
  Res. 108 (2003) 4025.

\bibitem{VeiNouKah06}
B.~Veihelmann, T.~Nousiainen, M.~Kahnert, W.~J. van~der Zande, Light scattering
  by small feldspar particles simulated using the {G}aussian random sphere
  geometry, J. Quant. Spectrosc. Radiat. Transfer 100 (2006) 393--405.

\bibitem{NouMcF04}
T.~Nousiainen, G.~M. McFarquhar, Light scattering by quasi-spherical ice
  crystals, J. Atmos. Sci. 61 (2004) 2229--2248.

\bibitem{Bogachev.V.I_1998_Gaussian_Measures}
V.~I. Bogachev, {{G}aussian Measures}, American Mathematical Society,
  Providence, 1998.

\bibitem{Lifshits.M.A_1995_Gaussian_Random_Functions}
M.~A. Lifshits, {{G}aussian Random Functions}, Kluwer Academic Publishers,
  Dordrecht, 1995.

\bibitem{sell2013dynamics}
G.~R. Sell, Y.~You, Dynamics of {E}volutionary {E}quations, Springer-Verlag,
  New York, 2013.

\end{thebibliography}

\end{document}